\documentclass[12pt, reqno]{amsart}

\usepackage{amsthm}
\usepackage{amsfonts}
\usepackage{amsmath}
\usepackage{mathrsfs}
\usepackage{amssymb}
\usepackage{amsfonts}
\usepackage{amsbsy}

\usepackage{CJK,CJKnumb}
\usepackage{color}
\usepackage{indentfirst}
\usepackage{latexsym,bm}
\usepackage{graphicx}
\usepackage{cases}
\usepackage{pifont}
\usepackage{txfonts}
\usepackage{xcolor}
\usepackage[all,knot,poly]{xy}

\usepackage[labelformat=empty]{caption}
\usepackage{enumitem}

\usepackage{ulem}

\newtheorem{theo}{Theorem}[section]
\newtheorem{lem}[theo]{Lemma}
\newtheorem{coro}[theo]{Corollary}
\newtheorem{prop}[theo]{Proposition}
\newtheorem{defi}[theo]{Definition}

\theoremstyle{remark}
\theoremstyle{rem}
\newtheorem{rem}[theo]{Remark}
\theoremstyle{ex}

\newtheorem{example}[theo]{Example}

\numberwithin{equation}{section}

\newcommand{\C}{{\mathbb C}}

\newcommand{\Z}{{\mathbb Z}}

\newcommand{\osp}{{\rm\mathfrak{osp}}}
\newcommand{\Sl}{{\rm\mathfrak{sl}}}

\newcommand{\g}{{\mathfrak g}}

\newcommand{\de}{\delta}
\newcommand{\ga}{\gamma}

\newcommand{\Q}{{\mathcal Q}}
\newcommand{\hQ}{{\widehat{\mathcal Q}}}

\newcommand{\cq}{{\mathcal{C}_q}}
\newcommand{\cqq}{{\mathfrak{C}_q}}

\newcommand{\Uq}{{{\rm  U}_q}}
\newcommand{\U}{{\rm  U }}
\newcommand{\Dr}{{\rm U}^D}

\newcommand{\D}{\mathfrak{U}^D}
\newcommand{\UU}{{\mathfrak U}}

\newcommand{\cI}{{\mathcal I}}

\newcommand{\e}{\xi^+ }
\newcommand{\f}{\xi^-}
\newcommand{\ef}{\xi^{\pm}}
\newcommand{\qe}{e}
\newcommand{\qf}{f}
\newcommand{\x}{\xi'}
\newcommand{\h}{\kappa}
\newcommand{\qk}{\gamma}
\newcommand{\hh}{\hat{\kappa}}

\newcommand{\X}{X}

\newcommand{\HH}{H}

\newcommand{\ka}{\mathfrak{k}}
\newcommand{\ta}{\mathfrak{t}}

\newcommand{\xp}{{\rm {exp}}}

\newcommand{\ad}{{ \mbox{Ad}}}

\newcommand{\up}{\Upsilon}

\newcommand{\fU}{{\mathfrak U}}

\allowdisplaybreaks

\begin{document}

\title[Drinfeld realisation and Vertex operator representations]{Drinfeld realisations and Vertex operator representations of \\ quantum affine superalgebras}

\author{Ying Xu}
\author{R.  B.  Zhang}

\address[Xu]{School of Mathematics, Hefei University of Technology, Anhui Province, 230009, China}
\email{xuying@hfut.edu.cn}
\address[Xu, Zhang]{School of Mathematics and Statistics,
University of Sydney, NSW 2006, Australia}
\email{ruibin.zhang@sydney.edu.au}

\begin{abstract}
Drinfeld realisations are constructed for the quantum affine superalgebras of the series $\osp(1|2n)^{(1)}$,$\Sl(1|2n)^{(2)}$ and $\osp(2|2n)^{(2)}$. By using the
realisations, we develop vertex operator representations and classify the finite dimensional irreducible representations for these quantum affine superalgebras.

\end{abstract}
\subjclass[2010]{81R10,17B37,17B69}
\keywords{Quantum affine superalgebras; Drinfeld realisations; vertex operator representations}

\maketitle

\tableofcontents

\section{Introduction}\label{intro}

Quantum supergroups associated with simple Lie superalgebras and their affine analogues 
were introduced \cite{BGZ, Y94, ZGB91b}  and extensively studied 
(see, e.g., \cite{Z93, Z98} and references therein)  in the 90s. 
They have important applications in a variety of areas, most notably, topology of knots and $3$-manifolds \cite{Z92a, Z95}, and the theory of integrable models of Yang-Baxter type \cite{BGZ, ZBG91}.  In particular,
finite dimensional representations of quantum affine superalgebras play a crucial
role in the latter area in constructing integrable models by solving the
$\Z_2$-graded Jimbo equations \cite{BGZ} to obtain solutions of the Yang-Baxter equation.
In recent years there has been a resurgence of interest in quantum supergroups and quantum affine superalgebras from the point of view algebra and representation theory.

In this paper,  we will construct the Drinfeld realisations, develop vertex operator representations and classify the irreducible finite dimensional representations for  the quantum affine superalgebras $\Uq(\g)$ associated with the following series of affine Lie superalgebras $\g$:
\begin{eqnarray}\label{eq:g}
\osp(1|2n)^{(1)},  \quad \Sl(1|2n)^{(2)}, \quad \osp(2|2n)^{(2)}, \quad n\ge 1.
\end{eqnarray}

We wish to point out that little is known about Drinfeld realisations, vertex operator representations, or the classification of irreducible finite dimensional representation for quantum affine superalgebras, except for untwisted type $A$. Even in this case, the study of vertex operator representations is not very systematic.  

The affine Lie superalgebras in \eqref{eq:g} do not have isotropic odd roots, thus have much similarity to ordinary affine Lie algebras (but we wre unable to find a proper treatment of their vertex operator representations). However, 
the quantum  affine superalgebras associated with these affine Lie superalgebras have some strikingly new features. In particular, they  admit irreducible integrable highest weight representations which do not have classical (i.e., $q\to 1$) counter parts. Some of the irreducible vertex operation representations constructed in this paper are of this kind. 

Below is a brief description of the main results of the paper and techniques used to prove them.

\smallskip
\noindent{\bf 1.1.}
The Drinfeld realisation of a quantum affine algebra \cite{Dr}
is a quantum analogue of the loop algebra realisation of an affine Lie algebra.
It is indispensable for studying vertex operator representations \cite{Jn1, JnM} and
finite dimensional representations \cite{CP1, CP2} of the quantum affine algebra.
The equivalence between the Drinfeld realisation and usual Drinfeld-Jimbo presentation in terms of Chevalley generators was known to Drinfeld \cite{Dr}, and has been investigated in a number of papers,  see, e.g.,  \cite{Be,De, Jn2, JZ}.

Previously Drinfeld realisations for quantum affine superalgebras were only known for untwisted types $A$ \cite{Y99} and $D(2, 1; \alpha)$ \cite{H} in the standard root systems.  The realisation in type $A$  formed the launching pad for the study of integrable representations of the quantum affine special linear superalgebra in \cite{WZ, Zh}.

In this paper, we construct the Drinfeld realisation for the quantum affine superalgebra associated with each of the affine Lie superalgebras in \eqref{eq:g}, see Definition \ref{def:DR}. We establish in Theorem \ref{them:DR-iso} an isomorphism between the quantum superalgebra of Definition \ref{def:DR} and the corresponding quantum affine superalgebra presented in the standard way by using Chevalley generators and defining relations \cite{Z2}.  As explained in Remark \ref{rem:hopf-iso}, the isomorphism in Theorem \ref{them:DR-iso} can in fact be interpreted as an
isomorphism of Hopf superalgebras. 

We prove Theorem \ref{them:DR-iso} by relating the Drinfeld realisations of the quantum
affine superalgebras to Drinfeld realisations of some ordinary quantum affine algebras,
and then applying Drinfeld's theorem \cite{Dr}. This makes essential use of
the notion of quantum correspondences introduced in \cite{XZ, Z2} (also see \cite{MW}).
A quantum correspondence between a pair $(\g, \g')$ of (affine) Lie superalgebras
is a Hopf superalgebra isomorphism between the corresponding quantum (affine) superalgebras.
Here we regard the category of vector superspaces as a braided tensor category, 
and a Hopf superalgebra as a Hopf algebra over this category. 
References \cite{XZ, Z2} contain a systematical treatment of quantum correspondences, 
some of which appeared as S-dualities in string theory in
work of Mikhaylov and Witten \cite{MW}. For convenience, we give in
Lemma \ref{lem:correspon} a concise description of the quantum correspondences 
used in this paper.

\smallskip
\noindent{\bf 1.2.} Apart from the case of untwisted type $A$, the construction of vertex operator representations and classification of irreducible finite dimensional representation for quantum affine superalgebras were hardly studied previously.  What hinders progress in the area is the lack of Drinfeld realisations  

By making use of the Drinfeld realisation obtained for the quantum affine superalgebras associated with the affine Lie superalgebras in \eqref{eq:g}, we construct vertex operator representations of the quantum affine superalgebras at level $1$, and also classify the finite dimensional irreducible representations. The main results are given in Theorem \ref{them:v.o} (and its variations in Sections \ref{sect:other-level-1}  and \ref{sect:other-vacuum}) and Theorem \ref{theo-finite module}.

The vertex operator representations of $\Uq(\g)$ constructed here are realised on quantum Fock spaces; they are level $1$ irreducible integrable highest weight representations relative to the standard triangular decomposition. [Recall that there exists a well defined notion of
integrable highest weight representations \cite{Z2} for $\Uq(\g)$ with $\g$ belonging to \eqref{eq:g},  even though this is not true for most of the other quantum affine superalgebras.]
Our construction here is heavily influenced by work of Jing \cite{Jn1, JnM} on  the vertex operator representations of ordinary quantum affine algebras. As we have mentioned already, some of the
vertex operator representations constructed here do not have classical (i.e., $q\to1$) limits.

The finite dimensional irreducible representations of $\Uq(\g)$ are shown to be level $0$ highest weight representations relative to another triangular decomposition.
We obtain the necessary and sufficient conditions on the highest weights for
irreducible highest weight representations to be finite dimensional. The conditions are described in terms of Drinfeld’s highest weight polynomials.  The proof of the classification theorem (Theorem \ref{theo-finite module}) makes essential use of results of Chari and Pressley in \cite{CP0, CP1, CP2} on ordinary quantum affine algebras. Another important ingredient in the proof is the quantum correspondences between affine Lie superalgebras discussed in Section \ref{sec:DR}.  

\medskip 
\noindent{\bf 1.3.} 
Finite dimensional representations of quantum affine superalgebras play a crucial role in constructing soluble models of Yang-Baxter type \cite{BGZ}; vertex  operator representations form an integral part of conformal field theory.  Thus results of this paper have direct applications in mathematical physics.

Throughout the paper,  $K:=\C(q^{1/2})$ is the field of rational functions in the indeterminate $q^{1/2}$.
For any element $z\ne 0$ or a root of $1$ in a field,
\[\begin{aligned}
&[0]_z=[0]_z!=1,\quad [k]_z=\frac{z^k-z^{-k}}{z-z^{-1}},\quad \mbox{for}\ \  k\in\Z^*, \\
&[N]_z!=\prod_{i=1}^N[i]_z, \quad \begin{bmatrix} N\\k\end{bmatrix}_z=\frac{[N]_z!}{[N-k]_z![k]_z!},\quad \mbox{for}\ \  k\leq N \in\Z_{>0}.
\end{aligned}
\]

\section{ Drinfeld realisations of quantum affine superalgebras}\label{quantum}
Consider the affine Kac-Moody superalgebras given in \eqref{eq:g}:
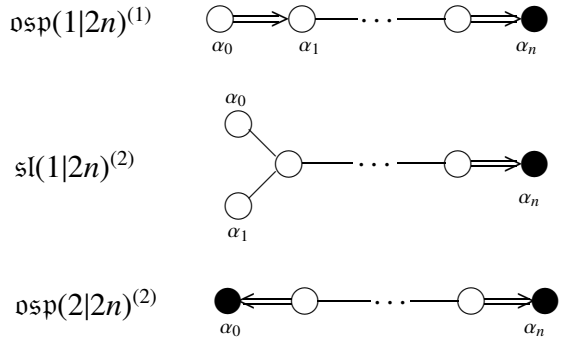
\[
\begin{picture}(120, 28)(40,-12)
\put(-42,-5) {$\osp(1|2n)^{(1)}$}
\put(38,0){\circle{10}}
\put(35,-12){\tiny\mbox{$\alpha_0$}}
\put(43,1){\line(1, 0){17}}
\put(43,-1){\line(1, 0){17}}
\put(57,-3){$>$}
\put(69, 0){\circle{10}}
\put(68,-12){\tiny\mbox{$\alpha_1$}}
\put(74, 0){\line(1, 0){16}}
\put(91, -1){\dots}
\put(105, 0){\line(1, 0){18}}
\put(128, 0){\circle{10}}
\put(133,1){\line(1, 0){17}}
\put(133,-1){\line(1, 0){17}}
\put(145,-3){$>$}
\put(157, 0){\circle*{10}}
\put(150,-12){\mbox{\tiny$\alpha_{n}$}}
\end{picture}
\]
\[
\begin{picture}(150, 60)(-20,-28)
\put(-85,-5) {$\Sl(1|2n)^{(2)}$}
\put(0, 15){\circle{10}}
\put(-5,23){\tiny$\alpha_0$}
\put(0, -16){\circle{10}}
\put(-4,-28){\tiny$\alpha_1$}
\put(14, -3){\line(-1, -1){10}}
\put(14, 3){\line(-1, 1){10}}
\put(19, 0){\circle{10}}
\put(24, 0){\line(1, 0){19}}
\put(45, -1){\dots}
\put(60,0){\line(1, 0){18}}
\put(83, 0){\circle{10}}
\put(88,1){\line(1, 0){17}}
\put(88,-1){\line(1, 0){17}}
\put(100,-3){$>$}
\put(112, 0){\circle*{10}}
\put(106,-15){\tiny$\alpha_{n}$}
\end{picture}
\]
\[
\begin{picture}(150, 30)(-10,-14)
\put(-75,-5){$\osp(2|2n)^{(2)}$}
\put(6,0){\circle*{10}}
\put(3,-12){\tiny $\alpha_0$}
\put(12,1){\line(1, 0){18}}
\put(12,-1){\line(1, 0){18}}
\put(9,-3){$<$}
\put(35, 0){\circle{10}}
\put(40, 0){\line(1, 0){20}}
\put(61, -1){\dots}
\put(75, 0){\line(1, 0){18}}
\put(98, 0){\circle{10}}
\put(103,1){\line(1, 0){17}}
\put(103,-1){\line(1, 0){17}}
\put(115,-3){$>$}
\put(126, 0){\circle*{10}}
\put(115,-12){ \tiny$\alpha_{n}$}
\end{picture}
\]
Here the notation is as in \cite{K78}, with $\osp(1|2n)^{(1)}$ denoting the untwisted affine  Lie superalgebra of $\osp(1|2n)$,  and $\osp(2|2n)^{(2)}$ and $\Sl(1|2n)^{(2)}$ the
twisted (by order two automorphisms) affine Lie superalgebras of $\osp(2|2n)$ and $\Sl(1|2n)$ respectively.
In the Dynkin diagrams, the black nodes denote the odd simple roots, while the white ones are even simple roots. More details on these root systems can be found in \cite{K78} (also see \cite{XZ, Z2}).

Let $\g$ be an affine Lie superalgebra in \eqref{eq:g}.
We denote by
$A=(a_{ij})$ its Cartan matrix, which is realised in terms of the set of simple roots
$\Pi=\{\alpha_i \mid i=0, 1, 2, \dots, n\}$
with $a_{i j} = \frac{2(\alpha_i, \alpha_j)}{(\alpha_i, \alpha_i)}$. A simple root
$\alpha_i$ is odd if the corresponding node in the Dynkin diagram is black, and is even otherwise. Set $q_i=q^{\frac{(\alpha_i,\alpha_i)}{2}}$ for all $\alpha_i\in\Pi$.

For any superalgebra $A=A_{\bar{0}}\bigoplus A_{\bar{1}}$, we define the parity $[\,]: A\rightarrow \Z_{2}=\{0, 1\}$ of homogeneous elements of $A$ as follows: $[a]={0}$ if $a\in A_{\bar{0}}$ and $[a]={1}$  if $a\in A_{\bar{1}}$.

\subsection{Drinfeld realisations}\label{sect:Dr-def}
Let us first recall the standard definition of quantum affine superalgebras given in terms of Chevalley generators and defining relations. 
\begin{defi}[\cite{Z2}]\label{defi:quantum-super}
Let $\g=\osp(1|2n)^{(1)}$, $\Sl(1|2n)^{(2)}$ or $\osp(2|2n)^{(2)}$.
The quantum affine  superalgebra $\Uq(\g)$ over ${K}$  is an associative superalgebra with identity generated by the homogeneous
elements $\qe_i,\qf_i, k_i^{\pm1}$ ($0\le i \le n$), where $\qe_s,\qf_s,(s\in\tau)$, are odd and the other generators are even, with the following defining relations:
\begin{eqnarray}
\nonumber
&& k_i  k_i^{-1}= k_i^{-1}  k_i=1,\quad  k_i  k_j= k_j  k_i,\\
\nonumber
&&     k_i \qe_j  k_i^{-1} = q_i^{a_{ij}}  \qe_j,
\quad  k_i \qf_j   k_i^{-1} = q_i^{-a_{ij}} \qf_j,\\
\label{eq:xx-q}
&&\qe_i\qf_j     -    (-1)^{ [\qe_i][\qf_j] } \qf_j\qe_i
                   =\de_{ij}  \dfrac{  k_i- k_i^{-1} }
                                          { q_i-q_i^{-1} }, \quad \forall i, j, \\
\nonumber
&&\left(
            \mbox{Ad}_{\qe_i}    \right)^{1-a_{ij}}    (\qe_j)
    =\left(
             \mbox{Ad}_{\qf_i}    \right)^{1-a_{ij}}    (\qf_j)=0, \quad       \text{ if } i\neq j.
\end{eqnarray}
\end{defi}

Here $\mbox{Ad}_{\qe_i}(x)$ and $\mbox{Ad}_{\qf_i}(x)$ are respectively defined by
\begin{equation}\label{eq:ad}
\begin{aligned}
\mbox{Ad}_{\qe_i}(x)     =         \qe_ix   -(-1)^{[\qe_i][x]}   k_i x  k_i^{-1} \qe_i,\\
\mbox{Ad}_{\qf_i}(x)      =         \qf_ix    -(-1)^{[\qf_i][x]}    k_i^{-1} x  k_i \qf_i,
\end{aligned}
\end{equation}
For any $x, y\in \Uq(\g)$ and $a\in{K}$, we shall write
\[
[x, y]_a=x y - (-1)^{[x][y]} a y x, \quad [x, y] = [x, y]_1.
\]
Then $\mbox{Ad}_{\qe_i}(\qe_j)=[\qe_i, \qe_j]_{q_i^{a_{ij}}}$ and
$\mbox{Ad}_{\qf_i}(\qf_j)=[\qf_i, \qf_j]_{q_i^{a_{ij}}}$.

To construct the Drinfeld realisation for  the quantum affine superalgebra $\U_q(\g)$, we let $\cI=\{ (i,r)\mid 1\le i\le n,  \ r\in\Z \}$.  Define the set $\cI_\g$ by
$\cI_{\g}:=\cI$
if $\g=\osp(1|2n)^{(1)}$ or $\Sl(1|2n)^{(2)}$; and
$\cI_{\g}:=\cI\backslash \{ (i,2r+1)\mid 1\le i<n, \  r\in \Z\}$
if  $\g=\osp(2|2n)^{(2)}$.
Let $ \mathcal{I}_{\g}^* =\{(i, s)\in \mathcal{I}_{\g}\mid s\ne 0\}$.
Also, for any expression $f(x_{r_1},\dots,x_{r_k})$ in $x_{r_1},\dots,x_{r_k}$, we use $sym_{r_1,\dots,r_k}f(x_{r_1},\dots,x_{r_k})$ to denote $\sum_{\sigma}f(x_{\sigma(r_1)},\dots,x_{\sigma(r_k)})$, where the sum is over the permutation group of the set  $\{r_1, r_2, \dots, r_k\}$.

\begin{defi}\label{def:DR}
For $\g=\osp(1|2n)^{(1)}$, $\Sl(1|2n)^{(2)}$ or $\osp(2|2n)^{(2)}$,  we let $\Dr_q(\g)$ be the associative superalgebra over ${K}$ with identity,  generated by
\[
\ef_{i,r},\ \qk_i^{\pm1}, \  \h_{i,s}, \ \ga^{\pm 1/2}, \quad \text{for }\  (i,r)\in\mathcal{I}_{\g},   \   (i,s)\in\mathcal{I}_{\g}^*, \  1\le i\le n,
\]
where $\e_{n,r},\f_{n,r}$ are odd and the other generators are even, with the following defining relations.
\begin{itemize}
\item[\rm(1)]  $\ga^{\pm 1/2}$ are central, and $\ga^{1/2} \ga^{- 1/2}=1$,
\begin{align}
&\qk_i\qk_i^{-1}=\qk_i^{-1}\qk_i=1,\quad  \qk_i\qk_j=\qk_j\qk_i, \nonumber\\
\label{eq:hx}
& \qk_i \ef_{j,r} \qk_i^{-1}=q_i^{\pm a_{ij}} \ef_{j,r},\quad
 [\h_{i,r},\ef_{j,s}]  = \dfrac{  u_{i,j,r} \ga^{\mp|r|/2}  }
                                                   {   r(q_i-q_i^{-1})  }
                                                    \ef_{j,s+r},  \\
  \label{eq:hh}
&[\h_{i,r},\h_{j,s}]=\delta_{r+s,0} \dfrac{ u_{i,j,r} (\ga^{r}-\ga^{-r})  }
                                                           { r   (q_i-q_i^{-1})(q_j-q_j^{-1})  }  ,\\
\label{eq:xx}
& [\e_{i,r}, \f_{j,s}]  =\delta_{i,j}
                                \dfrac{   \ga^{\frac{r-s}{2}} \hh^{+}_{i,r+s}
                                          -  \ga^{\frac{s-r}{2}} \hh^{-}_{i,r+s}   }
                                         {  q_i-q_i^{-1}  },
\end{align}
where the
$u_{i,j,r}$ are given in \eqref{eq:u-def};  and $\hh^{\pm}_{i,\pm r}$ are defined by
\begin{equation}\label{eq:hh-hat}
\begin{aligned}
&\sum_{r\in\Z}  \hh^{+}_{i,r}u^{-r} =\qk_i \xp  \left(
                                  (q_i-q_i^{-1})\sum_{r>0}\h_{i, r}u^{-r}
                                                               \right),\\
&\sum_{r\in\Z}  \hh^{-}_{i,-r}u^r    =\qk_i^{-1}  \xp  \left(
                                 (q_i^{-1}-q_i)\sum_{r>0}\h_{i,-r}u^r
                                                                \right);
\end{aligned}
\end{equation}

\item[\rm(2)] {\rm Serre relations}
\begin{itemize}
\item[\rm (A)] {\rm For}  $(\g,i,j)\neq (\osp(1|2n)^{(1)}, n, n)$,
 \begin{align}
 [\ef_{i,r\pm \theta}, \ef_{j,s}]_{q^{a_{ij}}_{i}}
     +[\ef_{j,s\pm \theta}, \ef_{i,r}]_{q^{a_{ji}}_{j}}=0, \label{eq:xrs-xsr}
 \end{align}
where $\theta=2$ if $\g=\osp(2|2n)^{(2)}, (i,j)\ne (n,n)$,  and  $1$ otherwise;
\item[\rm (B)] $n\ne i\neq j$, \ $\ell=1-a_{i j}$,
\[\begin{aligned}
\hspace{20mm}
sym_{r_1,\dots,r_\ell}\sum_{k=0}^\ell  (-1)^k
                                             \begin{bmatrix} \ell\\k\end{bmatrix}_{q_i}
                                              \ef_{i,r_1}\dots\ef_{i,r_k} \ef_{j,s}\ef_{i,r_k+1}\dots\ef_{i,r_\ell}=0;
\end{aligned}\]

\item[\rm (C)]  $j<n-1$, \ $\ell=1-a_{n j}$, and in case $\g\ne \Sl(1|2n)^{(2)}$, 
\[\begin{aligned}
\hspace{20mm}
sym_{r_1,\dots,r_\ell}\sum_{k=0}^\ell
                                              \begin{bmatrix} \ell\\k\end{bmatrix}_{\sqrt{-1} q_n}
                                              \ef_{n,r_1}\dots\ef_{n,r_k} \ef_{j,s}\ef_{n,r_k+1}\dots\ef_{n,r_\ell}=0;
\end{aligned}\]

\item[\rm (D)] {\rm For} $\g=\osp(1|2n)^{(1)}$,
 \begin{align*}
&sym_{r_1,r_2,r_3} [  [\ef_{n,r_1\pm 1},\ef_{n,r_2}]_{q_n^2},  \ef_{n,r_3}]_{q_n^4}=0;\\
&sym_{r,s}\Big([\ef_{n,r\pm 2}, \ef_{n,s}]_{q_n^2} -q_n^4 [\ef_{n,r\pm 1},  \ef_{n,s\pm 1} ]_{q_n^{-6}}\Big)=0;\\
&sym_{r,s}\Big(q_n^2[[\ef_{n,r\pm1},\ef_{n,s}]_{q_n^2},\ef_{n-1,k}]_{q_n^4}\\
&\qquad\ \  +(q_n^2+q_n^{-2})[[\ef_{n-1,k},\ef_{n,r\pm1}]_{q_n^2},\ef_{n,s}]\Big)=0;
\end{align*}

\item[\rm (E)] {\rm For} $\g=\osp(2|2n)^{(2)}$,
\[
sym_{r, s} [  [\ef_{n-1 ,k},\ef_{n, r\pm 1}]_{q_n^2},  \ef_{n, s}]=0.
\]
\end{itemize}
\end{itemize}
\end{defi}
In the above, the scalars $u_{i,j,r}$ ($r\in\Z$, $i, j=1, 2, \dots, n$)  are defined by
\begin{eqnarray}\label{eq:u-def}
\begin{aligned}
&\osp(1|2n)^{(1)}: \quad u_{i,j,r}=\begin{cases}
 q_n^{4r}-q_n^{-4r}-q_n^{2r}+q_n^{-2r},   & \text{if } i=j=n,\\
 q_i^{r a_{ij}}- q_i^{-r a_{ij}},                             & \text {otherwise };
            \end{cases}\\
&\osp(2|2n)^{(2)}: \quad u_{i,j,r}=\begin{cases}
  (-1)^r(q_n^{2r}-q_n^{-2r}),                 & \text{if }  i=j=n,  \\
(1+(-1)^r)(q_i^{r a_{ij}/2}-q_i^{-r a_{ij}/2}),        & \text {otherwise };
                  \end{cases}\\
&\Sl(1|2n)^{(2)}: \quad \phantom{X}  u_{i,j,r}=\begin{cases}
(-1)^r(q_n^{2r}-q_n^{-2r}),                   & \text{if }  i=j=n,  \\
q_i^{r a_{ij}}- q_i^{-r a_{ij}},        & \text {otherwise }.
                 \end{cases}
\end{aligned}
\end{eqnarray}

\subsection{The main theorem}
The following theorem is one of the main results of this paper; its proof 
will be given in Section \ref{sec:DR-osp1}.
\begin{theo}\label{them:DR-iso}
Let $\g=\osp(1|2n)^{(1)}$, $\Sl(1|2n)^{(2)}$ or $\osp(2|2n)^{(2)}$.
There exists a superalgebra isomorphism
$
\Psi: \Uq(\g)\stackrel{\sim}\longrightarrow \Dr_q(\g)
$
such that
\begin{align*}
\text {for } \g=&\osp(1|2n)^{(1)}:\\
&\qe_i \mapsto \e_{i,0}, \quad \qf_i\mapsto \f_{i,0}, \quad k_i \mapsto \qk_i, \quad  k^{-}_i\mapsto \qk^{-}_i,\quad \text {for } 1\le i\le n,\\
& \qe_0\mapsto\ad_{\f_{1,0}} \dots \ad_{\f_{n,0}}\ad_{\f_{n,0}}\ad_{\f_{n-1,0}} \dots \ad_{\f_{2,0}}(\f_{1,1}) \ga \qk^{-1}_{\g},\\
& \qf_0\mapsto c_{\g}\ga^{-1}\qk_{\g}\ad_{\e_{1,0}} \dots \ad_{\e_{n,0}}\ad_{\e_{n-1,0}} \dots \ad_{\e_{2,0}} (\e_{1,-1}),\quad  k_0\mapsto\ga \qk_{\g}^{-1},\\
\text {for } \g=& \Sl(1|2n)^{(2)}:\\
&\qe_i \mapsto \e_{i,0}, \quad \qf_i\mapsto \f_{i,0}, \quad k_i \mapsto \qk_i, \quad  k^{-}_i\mapsto \qk^{-}_i,\quad \text {for } 1\le i\le n,\\
& \qe_0\mapsto\ad_{\f_{2,0}} \dots \ad_{\f_{n,0}}\ad_{\f_{n,0}}\ad_{\f_{n-1,0}} \dots \ad_{\f_{2,0}}(\f_{1,1}) \ga \qk^{-1}_{\g},\\
& \qf_0\mapsto c_{\g}\ga^{-1}\qk_{\g}\ad_{\e_{2,0}} \dots \ad_{\e_{n,0}}\ad_{\e_{n-1,0}} \dots \ad_{\e_{2,0}} (\e_{1,-1}),\quad  k_0\mapsto\ga \qk_{\g}^{-1},\\
\text {for } \g=&\osp(2|2n)^{(2)}:\\
&\qe_i \mapsto \e_{i,0}, \quad \qf_i\mapsto \f_{i,0}, \quad k_i \mapsto \qk_i, \quad  k^{-}_i\mapsto \qk^{-}_i,\quad \text {for } 1\le i\le n,\\
&\qe_0\mapsto\ad_{\f_{1,0}} \dots \ad_{\f_{n-1,0}}(\f_{n,1}) \ga \qk^{-1}_{\g},\\
&\qf_0\mapsto c_{\g}\ga^{-1}\qk_{\g}\ad_{\e_{1,0}} \dots \ad_{\e_{n-1,0}}(\e_{n,-1}),
\quad k_0\mapsto\ga \qk_{\g}^{-1},
\end{align*}
where $\qk_{\g}$ is defined by
\begin{align*}
&\qk_{\g}=\begin{cases}
\qk_1^2\qk_2^2\dots \qk_{n}^2, & \g=\osp(1|2n)^{(1)},\\
\qk_1\qk_2^2\dots \qk_{n}^2,     &\g=\Sl(1|2n)^{(2)},\\
\qk_1\qk_2\dots \qk_n,                & \g=\osp(2|2n)^{(2)},
\end{cases}
\end{align*}
and $c_{\g}\in{K}$  is determined by \eqref{eq:xx-q}.
\end{theo}
\begin{rem}\label{rem:hopf-iso}
We can transcribe the Hopf superalgebra structure of $\Uq(\g)$ to
$\Dr_q(\g)$ using $\Psi$. For example, if $\Delta$ is the 
co-multiplication of $\Uq(\g)$, the co-multiplication of $\Dr_q(\g)$ is 
given by $(\Psi\otimes\Psi)\circ\Delta\circ\Psi^{-1}$. 
Then clearly $\Psi$ is an isomorphism of Hopf superalgebras.
\end{rem}

\subsection{ Drinfeld realisations in terms of
	currents}\label{sec:DR-C}

For the purpose of studying vertex operator representations, it is more convenient to present the Drinfeld realisation in terms of
currents.  For this, we will need the calculus of formal distributions familiar in the theory of vertex operators algebras.  Particularly useful is the formal distribution
$\delta(z)=\sum_{r\in\Z}z^r$, which has the following property: for any formal distribution $f(z,w)$ in the two variables $z$ and $w$, we have
$f(z,w)\delta(\frac{w}{z})=f(z,z)\delta(\frac{w}{z})$. A detailed treatment of formal distributions can be found in, e.g., \cite{K98}.

Given any pair of simple roots $\alpha_i$ and $\alpha_j$ of $\g$, we let
\begin{align}\label{eq-g}
g_{ij}(z)=\sum_{n\ge 0} g_{ij,n}z^n,
\end{align}
be the Taylor series expansion at $z=0$ of $f_{ij}(z)/h_{ij}(z)$,
where
\begin{eqnarray*}
	&\osp(1|2n)^{(1)}: &f_{ij}(z)=\begin{cases}
		(q^{2(\alpha_i,\alpha_j)}z-1)(q^{-(\alpha_i,\alpha_j)}z-1), & i=j=n,\\
		q^{(\alpha_i,\alpha_j)}z-1, & otherwise;
	\end{cases}\\
	&                           &h_{ij}(z)=\begin{cases}
		(z-q^{2(\alpha_i,\alpha_j)})(z-q^{-(\alpha_i,\alpha_j)}),&i=j=n,\\
		z-q^{(\alpha_i,\alpha_j)},& otherwise;
	\end{cases}\\
	&\osp(2|2n)^{(2)}:&f_{ij}(z)=\begin{cases}
		(-q)^{(\alpha_i,\alpha_j)}z-1, & i=j=n,\\
		\left(q^{(\alpha_i,\alpha_j)/2}z-1\right)\left ((-q)^{(\alpha_i,\alpha_j)/2}z-1\right), & otherwise;
	\end{cases}\\
	&                           &h_{ij}(z)=\begin{cases}
		z-(-q)^{(\alpha_i,\alpha_j)},&i=j=n,\\
		\left(z-q^{(\alpha_i,\alpha_j)/2}\right)\left(z-(-q)^{(\alpha_i,\alpha_j)/2}\right),& otherwise;
	\end{cases}\\
	&\Sl(1|2n)^{(2)}: &f_{ij}(z)=\begin{cases}
		(-q)^{(\alpha_i,\alpha_j)}z-1, & i=j=n,\\
		q^{(\alpha_i,\alpha_j)}z-1, & otherwise;
	\end{cases}\\
	&                           &h_{ij}(z)=\begin{cases}
		z-(-q)^{(\alpha_i,\alpha_j)},&i=j=n,\\
		z-q^{(\alpha_i,\alpha_j)},& otherwise.
	\end{cases}
\end{eqnarray*}

Now we introduce the following formal distributions in $\Uq(\g)[[z^{1/2}, z^{-1/2}]]$ for $\g=\osp(1|2n)^{(1)}$ and $\Uq(\g)[[z, z^{-1}]]$ for $\g=\Sl(1|2n)^{(2)},\osp(2|2n)^{(2)}$,
\begin{align*}
&\e_i(z)=\begin{cases}
\sum_{r\in\Z}\e_{i,r}z^{-r+1/2}, & \g=\osp(1|2n)^{(1)},i=n;\\
\sum_{r\in\Z}\e_{i,r}z^{-r},& \text{otherwise},
\end{cases}\\
&\f_i(z)=\begin{cases}
\sum_{r\in\Z}\f_{i,r}z^{-r-1/2},& \g=\osp(1|2n)^{(1)},i=n;\\
\sum_{r\in\Z}\f_{i,r}z^{-r},& \text{otherwise},
\end{cases}
\\
&\psi_i(z)=\sum_{r\in\Z_{\ge 0}}\hh^{+}_{i,r}z^{-r},\quad \varphi_i(z)=\sum_{r\in\Z_{\le 0}}\hh^{-}_{i,r}z^{-r}.
\end{align*}

\begin{lem}\label{lem:dr-f} Let
	$\g=\osp(1|2n)^{(1)}$, $\Sl(1|2n)^{(2)}$ or $\osp(2|2n)^{(2)}$. Then $\Uq(\g)$ has the following presentation. The generators are
	\[
	\ef_{i,r},\hh_{i,r}^{\pm}, \ \ga^{\pm 1/2},
	\quad \text{for }\  (i,r)\in\mathcal{I}_{\g};
	\]
	the relations in terms of formal distributions are given by:
	\begin{item}
		\item[\rm (1)]  $\ga^{\pm 1/2}$ are central with $\ga^{1/2} \ga^{- 1/2}=1$,
		\begin{align}
		&\hh^{+}_{i,0}\hh^{-}_{i,0}= \hh^{-}_{i,0}\hh^{+}_{i,0}=1,
		[\varphi_i(z), \psi_j(w)]=[\psi_j(w),\varphi_i(z)]=0,\\
		& \varphi_i(z) \psi_j(w) \varphi_i(z)^{-1} \psi_j(w)^{-1}=g_{ij}(zw^{-1}\ga^{-1})/g_{ij}(zw^{-1}\ga),\\
		& \varphi_i(z) \ef_j(w) \varphi_i(z)^{-1}=g_{ij}(zw^{-1}\ga^{\mp1/2})^{\pm1}\ef_j(w),\\
		&\psi_i(z) \ef_j(w) \psi_i(z)^{-1}=g_{ij}(z^{-1}w\ga^{\mp1/2})^{\mp1}\ef_j(w),\\
		&[\e_i(z), \f_j(w)]=
		\frac{\rho_{z,w}\delta_{i,j}}{q_i-q_i^{-1}}
		\left (\psi_i(z\ga^{-1/2}) \delta\left( \frac{z\ga^{-1}}{w}\right)
		-\varphi_i(z\ga^{1/2}) \delta\left(\frac{z\ga}{w}\right) \right),\label{eq:xx-f}
		\end{align}
		where $g_{ij}$ are defined by \eqref{eq-g}, and
		$\rho_{z,w}=(z/w)^{1/2}$ if $\g=\osp(1|2n)^{(1)}$ and $i=n$,  and $\rho_{z,w}=1$ otherwise.
		
		\item[\rm (2)] Serre relations
		
		\begin{itemize}
			\item[\rm (A)] $(i,j)\neq (n, n)$, and if $\g\neq \osp(1|2n)^{(1)}$,
			\begin{align}
			\label{eq:xrs-xsr-f}
			&[z^{\pm\theta}\ef_{i}(z),\ef_{j}(w)]_{q_{i}^{a_{ij}}}+[w^{\pm\theta}\ef_{j}(w),\ef_{i}(z)]_{q_{j}^{a_{ji}}}=0,
			\end{align}
			where $\theta=2$ if $\g=\osp(2|2n)^{(2)}$,  and $1$ if $\g=\Sl(1|2n)^{(2)}$;
			
			\item[\rm (B)] $n\ne i\neq j$, \ $\ell=1-a_{i j}$,
			\[\begin{aligned}
			\hspace{10mm}
			sym_{z_1,\dots,z_\ell}\sum_{k=0}^\ell  (-1)^k
			\begin{bmatrix} \ell\\k\end{bmatrix}_{q_i}
			\ef_{i}(z_1)\dots\ef_{i}(z_k) \ef_{j}(w)\ef_{i}(z_{k+1})\dots\ef_{i}(z_{\ell})=0;
			\end{aligned}\]
			
			\item[\rm (C)]  $n=i\ne j$, \ $\ell=1-a_{i j}$, and if $\g\ne \Sl(1|2n)^{(2)}$, $j<n-1$,
			\[\begin{aligned}
			\hspace{10mm}
			sym_{z_1,\dots,z_\ell}\sum_{k=0}^\ell
			\begin{bmatrix} \ell\\k\end{bmatrix}_{\tilde{q_i}}
			\ef_{i}(z_1)\dots\ef_{i}(z_k) \ef_{j}(w)\ef_{i}(z_{k+1})\dots\ef_{i}(z_\ell)=0,
			\end{aligned}\]
			where $\tilde{q_i}=(-1)^{1/2}q_i$;
			
			\item[\rm (D)] for $\g=\osp(1|2n)^{(1)}$,
			\begin{align*}
			&sym_{z_1,z_2,z_3} \left[  [z_1^{\pm}\ef_{n}(z_1),\ef_{n}(z_2)]_{q_n^2},  \ef_{n}(z_3)\right]_{q_n^4}=0;\\
			&sym_{z,w}\Big([z^{\pm 2}\ef_{n}(z), \ef_{n}(w)]_{q_n^2} -q_n^4 [z^{\pm}\ef_{n}(z),  w^{\pm 1}\ef_{n}(w) ]_{q_n^{-6}}\Big)=0;\\
			&sym_{z_1,z_2}\Big(q_n^2\left[[z_1^{\pm}\ef_{n}(z_1),\ef_{n}(z_2)]_{q_n^2},\ef_{n-1}(w)\right]_{q_n^4}\\
			&+(q_n^2+q_n^{-2})\left[[\ef_{n-1}(w),z_1^{\pm}\ef_{n}(z_1)]_{q_n^2},\ef_{n}(z_2)\right]\Big)=0;
			\end{align*}
			
			\item[\rm (E)] for $\g=\osp(2|2n)^{(2)}$,
			\[
			sym_{z_1,z_2} \left[  [\ef_{n-1}(w),z_1^{\pm}\ef_{n}(z_1)]_{q_n^2},  \ef_{n}(z_2)\right]=0.
			\]
		\end{itemize}
	\end{item}
\end{lem}
\begin{proof}
	This can be proven by straightforward computation, thus we will not give the details. Instead, we consider only \eqref{eq:xx-f} as an example. The relations \eqref{eq:xx} and \eqref{eq:hh-hat} lead to
	\begin{align*}
	&[\e_i(z), \f_j(w)]
	=\rho_{z,w}\delta_{i,j}\sum_{r,s}   \dfrac {   \ga^{\frac{r-s}{2}}  \hh^{+}_{i,r+s}
		-\ga^{\frac{s-r}{2}}  \hh^{-}_{i,r+s}     }
	{  q_i -q_i ^{-1}  }
	z^{-r}w^{-s}\\
	&=\frac{\rho_{z,w}\delta_{i,j}}{q_i-q_i^{-1}}\left\{
	\sum_{r}\hh^{+}_{i,r}       (z\ga^{-1/2})^{-r}    \delta\left(\frac{z\ga^{-1}}{w}\right)
	-\sum_{r}\hh^{-}_{i,r}  (z\ga^{1/2})^{-r}    \delta\left(\frac{z\ga}{w}\right)
	\right\}\\
	&=\frac{\rho_{z,w}\delta_{i,j}}{q_i-q_i^{-1}}\left\{
	k_i  \xp\left(
	(q_i -q_i^{-1}) \sum_{r=1}^{\infty}\h_{i,r}(z\ga^{-1/2})^{-r}
	\right)
	\delta \left( \frac{z\ga^{-1}}{w} \right) \right.\\
	&\hspace{21mm} \left.
	-k_i^{-1} \xp \left(
	(q_i^{-1}-q_i) \sum_{r=1}^{\infty}\h_{i,-r}(z\ga^{1/2})^{r}
	\right)
	\delta\left(\frac{z\ga}{w}\right)          \right\}.
	\end{align*}
	Using the definitions of $\psi_i(z)$ and $\varphi_i(z)$ on the right hand side, we immediately obtain \eqref{eq:xx-f}. In the opposite direction, we easily obtain
	\eqref{eq:xx} and \eqref{eq:hh-hat} by comparing the coefficients
	of $z^{-r}w^{-s}$ in \eqref{eq:xx-f}.
\end{proof}

\section{Quantum correspondences for Drinfeld realisations}\label{sec:DR}
We prove Theorem \ref{them:DR-iso} in this section. For convenience, We choose the normalisation for the bilinear form such that
$(\alpha_n,\alpha_n)=1.$ Consider the automorphism which leaves the other generators intact but maps $\h_{i,s}\mapsto [(\alpha_i,\alpha_i)/2]_q\h_{i, s}$ and $\e _{i,s}\mapsto [(\alpha_i,\alpha_i)/2]_q \e_{i, s}$ for all $i$.
It transforms the relations \eqref{eq:xx-q}, \eqref{eq:hx}-\eqref{eq:hh} into the following form
	\begin{eqnarray*}
		\begin{aligned}
			&\qe_i\qf_j     -    (-1)^{ [\qe_i][\qf_j] } \qf_j\qe_i
			=\de_{ij}  \dfrac{  k_i- k_i^{-1} }
			{ q-q^{-1} }, \quad \forall i, j, \\
			& [\h_{i,r},\ef_{j,s}]  = \dfrac{  u_{i,j,r} \ga^{\mp|r|/2}  }
			{   r(q-q^{-1})  }
			\ef_{j,s+r},  \\
			&[\h_{i,r},\h_{j,s}]=\delta_{r+s,0} \dfrac{ u_{i,j,r} (\ga^{r}-\ga^{-r})    }
			{  r   (q-q^{-1})(q-q^{-1})  }           ,\\
			& [\e_{i,r}, \f_{j,s}]  =\delta_{i,j}
			\dfrac{   \ga^{\frac{r-s}{2}}  \hh^{+}_{i,r+s}
				-  \ga^{\frac{s-r}{2}} \hh^{-}_{i,r+s}   }
			{  q-q^{-1}  }, \\
			&\sum_{r\in\Z}  \hh^{+}_{i,r}u^{-r} =\qk_i \xp  \left(
			(q-q^{-1})\sum_{r>0}\h_{i, r}u^{-r}
			\right),\\
			&\sum_{r\in\Z}  \hh^{-}_{i,-r}u^r    =\qk_i^{-1}  \xp  \left(
			(q^{-1}-q)\sum_{r>0}\h_{i,-r}u^r
			\right),
		\end{aligned}
	\end{eqnarray*}
In this section, we will use this form instead of the standard expressions in Definition \ref{defi:quantum-super} and \ref{def:DR}. A consequence is that $q^{\pm 1/2}$ never appears in these relations (see more details in \cite{XZ}).

\subsection{Smash products}\label{sect:smash-prod}
Let $\g$ be any of the affine Lie superalgebras
in the first row of Table \ref{tbl:g} (which is \eqref{eq:g}),

\begin{table}[h]
\caption{Table 1} \label{tbl:g}
\renewcommand{\arraystretch}{1.2}
\vspace{-2mm}
\begin{tabular}{c|c|c|c}
\hline
$\g$ & $\osp(1|2n)^{(1)}$    & $\Sl(1|2n)^{(2)}$ & $\osp(2|2n)^{(2)}$  \\
\hline
$\g'$ & $A_{2n}^{(2)}$   & $B_n^{(1)}$   & $D_{n+1}^{(2)}$   \\
\hline
\end{tabular}
\end{table}
and let $\g'$ be the ordinary affine Lie algebra corresponding to $\g$
in the second row.
We will speak about the pair $(\g, \g')$ of affine Lie (super)algebras in the table.
Now $\g'$ has the same
Cartan matrix  $A=(a_{ij})$ as $\g$. We let $\Pi'=\{\alpha'_0, \alpha'_1, \dots, \alpha'_n\}$ be the set of simple roots of $\g'$ which realises the Cartan matrix, and take $(\alpha'_n, \alpha'_n)=1$.  Let $t^{1/2}=\sqrt{-1} q^{1/2}$ and $t_i=t^{(\alpha'_i,\alpha'_i)/2}$ for all $i$.

The quantum affine  algebra  $\U_{t}(\g')$ is an associative  algebra over ${K}$ with identity generated by the elements $\qe'_i,\qf'_i, {k'_i}^{\pm1}$ ($0\le i \le n$) with the following defining relations:
\begin{eqnarray}
\nonumber
&& k'_i  {k'_i}^{-1}= {k'_i}^{-1}  k'_i=1,\quad  k'_i  k'_j= k'_j  k'_i,\\
\nonumber
&&     k'_i \qe'_j  {k'_i}^{-1} = t_i^{a_{ij}}  \qe'_j,
\quad  k'_i \qf'_j  {k'_i}^{-1}= t_i^{-a_{ij} }  \qf'_j,\\
\label{eq:ef-A}
&&\qe'_i\qf'_j     -   \qf'_j\qe'_i
                   =\de_{i,j} \dfrac{  k'_i- {k'_i}^{-1} }
                                          { t-t^{-1} }, \quad \forall i, j; \\
\nonumber
&&\left(
            \mbox{Ad}_{\qe'_i}    \right)^{1-a_{ij}}    (\qe'_j)
    =\left(
             \mbox{Ad}_{\qf'_i}    \right)^{1-a_{ij}}    (\qf'_j)=0, \quad       \text{ if } i\neq j.
\end{eqnarray}
$\mbox{Ad}_{\qe'_i}(x)$ and $\mbox{Ad}_{\qf'_i}(x)$ are respectively defined by
\begin{equation*}
\begin{aligned}
&\mbox{Ad}_{\qe'_i}(x)     =         \qe'_ix   -     k'_i x  {k'_i}^{-1} \qe'_i,\\
&\mbox{Ad}_{\qf'_i}(x)      =         \qf'_ix    -    {k'_i}^{-1} x  k'_i \qf'_i.
\end{aligned}
\end{equation*}
It is well known that $\Uq(\g')$ is a Hopf algebra.

Let $\mathcal{I}_{\g'}=\mathcal{I}_{\g}$ and $\mathcal{I}_{\g'}^*=\mathcal{I}_{\g}^*$
($\mathcal{I}_{\g}$ and $\mathcal{I}_{\g}^*$ are defined before
Definition \ref{def:DR}). The Drinfeld realisation $\Dr_{t}(\g')$ of $\U_t(\g')$  is an associative algebra over ${K}$ with identity generated by the generators
\[
\x^{\pm}_{i,r},\ \qk'^{\pm1}_i, \  \h'_{i,s}, \ \ga'^{\pm 1/2}, \quad \text{for }\  (i,r)\in\mathcal{I}_{\g'},   \   (i,s)\in\mathcal{I}_{\g'}^*, \  1\le i\le n,
\]
with the defining relations \cite{Dr}:

\begin{enumerate}
\item[(1)] $[{\ga'}^{\pm 1/2},\x^{\pm}_{i,r}]  =  [{\ga'}^{\pm 1/2},\h'_{i,r}]=[{\ga'}^{\pm 1/2},\qk'_{i}]=0$,
\begin{eqnarray}
\nonumber
&& \qk'_i {\qk'_i}^{-1}  = {\qk'_i}^{-1} \qk'_i=1, \quad \qk'_i \qk'_j=\qk'_j \qk'_i,\\
\nonumber
&& \qk'_i \x^{\pm}_{j,r} {\qk'_i}^{-1}  = t_i^{\pm a_{ij}} \x^{\pm}_{j,r},\quad
[\h'_{i,r},\x^{\pm}_{j,s}] = \dfrac
                                               {u'_{i,j,r}}
                                               {r(t -t^{-1})}
                                        {\ga'}^{\mp|r|/2}  \x^{\pm}_{j,s+r},\\
\nonumber
&&[\h'_{i,r},\h'_{j,s}] = \delta_{r+s,0} \dfrac
                                                         {   u'_{i,j,r}
                                                               (  \ga'^{r}- \ga'^{-r}  )   }
                                                        {      r(t -t ^{-1})(t -t ^{-1})    },\\
\nonumber
&&[\x^{-}_{i,r}, \x^{+}_{j,s}] = \de_{ij}   \dfrac
                                                           {    \ga'^{\frac{r-s}{2}}           \hh'^{+}_{i,r+s}
                                                             -  \ga'^{\frac{s-r}{2}}   \hh'^{-}_{i,r+s}    }
                                                           {t-t^{-1} },
\end{eqnarray}
where $\hh'^{\pm}_{i,\pm r}$ are defined by
\begin{equation}\label{eq:hh'}
\begin{aligned}
&\sum_{r\in\Z}  \hh'^{+}_{i,r}u^{-r}  =  \qk'_i  \xp\left(
                                         (t-t^{-1})  \sum_{r>0} \h'_{i, r} u^{-r}
                                                               \right),\\
&\sum_{r\in\Z}  \hh'^{-}_{i,-r}u^r  = \qk'^{-1}_i  \xp\left(
                                        (t^{-1}-t)   \sum_{r>0}  \h'_{i,-r} u^r
                                                            \right),
\end{aligned}
\end{equation}
and the scalars $u'_{i,j,r}$ are given in \eqref{eq:u'};

\item[\rm(2)] {\rm Serre relations}
\begin{itemize}
\item [(A)]	For $(\g,i,j)\neq (A_{2n}^{(2)}, n, n),$
\begin{align}\label{eq:sym-dr}
[\x^{\pm}_{i,r\pm \theta}, \x^{\pm}_{j,s}]_{t^{a_{ij}}_{i}}
+[\x^{\pm}_{j,s\pm \theta}, \x^{\pm}_{i,r}]_{t^{a_{ji}}_{j}}
=0,
\end{align}
where $\theta=2$ if $\g=D_{n+1}^{(2)}$,  $(i,j)\ne (n,n)$,  and $1$ otherwise;
	
\item[(B)] $n\ne i\neq j$, or $\g'\neq D_{n+1}^{(2)}$, $j+1<i=n$,\ $\ell=1-a_{i j}$ ,
\[\begin{aligned}
sym_{r_1,\dots,r_\ell}\sum_{k=0}^\ell  (-1)^k
      \begin{bmatrix} \ell\\k\end{bmatrix}_{t}
     \x^{\pm}_{i,r_1}\dots\x^{\pm}_{i,r_k}\x^{\pm}_{j,s}\x^{\pm}_{i,r_k+1}
                              \dots\x^{\pm}_{i,r_\ell}=0;
\end{aligned}\]

\item[\rm (C)] {\rm For} $\g=A_{2n}^{(2)}$,
 \begin{align*}
&sym_{r_1,r_2,r_3} [  [\ef_{n,r_1\pm 1},\ef_{n,r_2}]_{t_n^2},  \ef_{n,r_3}]_{t_n^4}=0;\\
&sym_{r,s}\Big([\ef_{n,r\pm 2}, \ef_{n,s}]_{t_n^2} -t_n^4 [\ef_{n,r\pm 1},  \ef_{n,s\pm 1} ]_{t_n^{-6}}\Big)=0;\\
&sym_{r,s}\Big(t_n^2[[\ef_{n,r\pm1},\ef_{n,s}]_{t_n^2},\ef_{n-1,k}]_{t_n^4}\\
&+(t_n^2+t_n^{-2})[[\ef_{n-1,k},\ef_{n,r\pm1}]_{t_n^2},\ef_{n,s}]\Big)=0;
\end{align*}

\item[\rm (D)] {\rm For} $\g=D_{n+1}^{(2)}$,
\[
sym_{r, s} [  [\ef_{n-1 ,k},\ef_{n, r\pm 1}]_{t_n^2},  \ef_{n, s}]=0.
\]
\end{itemize}
\end{enumerate}
The scalars $u'_{i,j,r}$ in the above equations are defined by
\begin{equation}\label{eq:u'}
\begin{aligned}
&A_{2n}^{(2)}: \quad u'_{i,j,r}=\begin{cases}
( t_n^{2r}-t_n^{-2r})(t_n^{2r}+t_n^{-2r}+(-1)^{r-1}),   & \text{if } i=j=n,\\
 t_i^{r a_{ij}}- t_i^{-r a_{ij}},                                        & \text {otherwise };
            \end{cases}\\
&B_n^{(1)}: \quad \phantom{X}  u'_{i,j,r}=t_i^{r a_{ij}}- t_i^{-r a_{ij}};\\
&D_{n+1}^{(2)}: \quad u'_{i,j,r}=\begin{cases}
t_n^{2r}-t_n^{-2r},                 & \text{if }  i=j=n,  \\
(1+(-1)^r)(t_i^{r a_{ij}/2}-t_i^{-r a_{ij}/2}),        & \text {otherwise }.
                  \end{cases}
\end{aligned}
\end{equation}



Applied to the quantum affine algebras under consideration, Drinfeld's  theorem \cite{Dr}
gives the following algebra isomorphism
\begin{eqnarray}\label{eq:JD-Dr-alg}
\rho:\U_{t}(\g')\stackrel{\sim}{\longrightarrow}\Dr_{t}(\g');
\end{eqnarray}
\begin{equation}\label{eq:iso-alg}
\begin{aligned}
\text {for } \g'=&A_{2n}^{(2)}:\\
&\qe'_i \mapsto \x^{+}_{i,0}, \quad \qf'_i\mapsto \x^{-}_{i,0}, \quad k'_i \mapsto \qk'_i, \quad  k'^{-}_i\mapsto \qk'^{-}_i,\quad \text {for } 1\le i\le n,\\
& \qe'_0\mapsto\ad_{\x^{-}_{1,0}} \dots \ad_{\x^{-}_{n,0}}\ad_{\x^{-}_{n,0}}\ad_{\x^{-}_{n-1,0}} \dots \ad_{\x^{-}_{2,0}}(\x^{-}_{1,1}) \ga' \qk'^{-1}_{\g'},\\
& \qf'_0\mapsto c_{\g'}\ga'^{-1}\qk'_{\g'}\ad_{\x^{+}_{1,0}} \dots \ad_{\x^{+}_{n,0}}\ad_{\x^{+}_{n-1,0}} \dots \ad_{\x^{+}_{2,0}} (\x^{+}_{1,-1}),\quad  k'_0\mapsto\ga' \qk'^{-1}_{\g'},\\
\text {for } \g'=& B_n^{(1)}:\\
&\qe'_i \mapsto \x^{+}_{i,0}, \quad \qf'_i\mapsto \x^{-}_{i,0}, \quad k'_i \mapsto \qk'_i, \quad  k'^{-}_i\mapsto \qk'^{-}_i,\quad \text {for } 1\le i\le n,\\
& \qe'_0\mapsto\ad_{\x^{-}_{2,0}} \dots \ad_{\x^{-}_{n,0}}\ad_{\x^{-}_{n,0}}\ad_{\x^{-}_{n-1,0}} \dots \ad_{\x^{-}_{2,0}}(\x^{-}_{1,1}) \ga' \qk'^{-1}_{\g'},\\
& \qf'_0\mapsto c_{\g'}\ga'^{-1}\qk'_{\g'}\ad_{\x^{+}_{2,0}} \dots \ad_{\x^{+}_{n,0}}\ad_{\x^{+}_{n-1,0}} \dots \ad_{\x^{+}_{2,0}} (\x^{+}_{1,-1}),\quad  k'_0\mapsto\ga' \qk'^{-1}_{\g'},\\
\text {for } \g'=&D_{n+1}^{(2)}:\\
&\qe'_i \mapsto \x^{+}_{i,0}, \quad \qf'_i\mapsto \x^{-}_{i,0}, \quad k'_i \mapsto \qk'_i, \quad  k'^{-}_i\mapsto \qk'^{-}_i,\quad \text {for } 1\le i\le n,\\
&\qe'_0\mapsto\ad_{\x^{-}_{1,0}} \dots \ad_{\x^{-}_{n-1,0}}(\x^{-}_{n,1}) \ga' \qk'^{-1}_{\g'},\\
&\qf'_0\mapsto c_{\g'}\ga'^{-1}\qk'_{\g'}\ad_{\x^{+}_{1,0}} \dots \ad_{\x^{+}_{n-1,0}}(\x^{+}_{n,-1}),
\quad k'_0\mapsto\ga' \qk'^{-1}_{\g'},
\end{aligned}
\end{equation}
where $\qk'_{\g'}$ is defined by
\begin{align*}
&\qk'_{\g}=\begin{cases}
\qk'^2_1\qk'^2_2\dots \qk'^2_{n}, & \g'=A_{2n}^{(2)},\\
\qk'_1\qk'^2_2\dots \qk'^2_{n},     &\g'=B_{n}^{(1)},\\
\qk'_1\qk'_2\dots \qk'_n,                & \g'=D_{n+1}^{(2)};
\end{cases}
\end{align*}
and $c_{\g}\in{K}$  can be fixed by \eqref{eq:ef-A}.

To prove Theorem \ref{them:DR-iso}, we will need to enlarge the quantum affine superalgebra $\Uq(\g)$
and the Drinfeld superalgebra $\Dr_q(\g)$  following \cite{XZ,Z2}.

Corresponding to each simple root  $\alpha_i$ of $\g$ for $i\neq 0$, we introduce a  group $\Z_2$ generated by $\sigma_i$ such that ${\sigma_i}^2=1$, and let $\mathrm{G}$ be the direct product of all such groups.
The group algebra ${K}\mathrm{G}$ has a standard Hopf algebra structure with the
co-multiplication given by $\Delta(\sigma_i)=\sigma_i\otimes\sigma_i$ for all $i$.
Define a left $\mathrm{G}$-action on $\Uq(\g)$ by
\begin{align}
\sigma_i\cdot e_j=(-1)^{(\alpha_i,\alpha_j)}e_j, \quad \sigma_i\cdot f_j=(-1)^{(\alpha_i,\alpha_j)}f_j, \quad \sigma_i\cdot k_j=k_j, \quad\text{$i\ne 0$},
\end{align}
which preserves the multiplication of $\Uq(\g)$.  This  defines a  left ${K}\mathrm{G}$-module algebra structure on $\Uq(\g)$. Similarly, let $\mathrm{G}$ act on $\Dr_q(\g)$ by
\begin{align}\label{eq:G-act}
&\sigma_i\cdot\xi^\pm_{j, r}=(-1)^{(\alpha_i,\alpha_j)}\xi^\pm_{j, r},
\quad \sigma_i\cdot\kappa_{j, r}= \kappa_{j, r},
\quad  \sigma_i\cdot\gamma_j=\gamma_j,  \quad  \sigma_i\cdot\gamma=\gamma,
\end{align}
for all $i, j\ge 1$ and $r\in\Z$. This again preserves the multiplication of  $\Dr_q(\g)$.

By using a standard construction in the theory of Hopf algebras,  we manufacture the smash product superalgebras
\begin{eqnarray}
\UU_q(\g):=\Uq(\g)\sharp{K}\mathrm{G}, \quad \D_q(\g):=\Dr_q(\g)\sharp{K}\mathrm{G},
\end{eqnarray}
which have underlying vector superspaces  $\Uq(\g)\otimes{K}\mathrm{G}$
and $\Dr_q(\g)\otimes{K}\mathrm{G}$ respectively,
where ${K}\mathrm{G}$ is regarded as purely even.
The multiplication of  $\UU_q(\g)$ (resp. $\D_q(\g)$) is defined,
for all $x, y$ in  $\Uq(\g)$ (resp. $\D_q(\g)$)  and $\sigma, \tau\in \mathrm{G}$, by
\[
(x\otimes \sigma)(y\otimes \tau)= x \sigma.y\otimes \sigma\tau.
\]
We will write $x \sigma$ and $\sigma x$ for $x\otimes \sigma$ and  $(1\otimes\sigma)(x\otimes 1)$ respectively.

In exactly the same way, we introduce a  group $\Z_2$ corresponding to each simple root  $\alpha'_i$ of $\g'$ with $i\neq 0$. The group is generated by $\sigma'_i$ such that ${\sigma'_i}^2=1$.
Let $\mathrm{G}'$ be the direct product of all such groups, and define a $\mathrm{G}'$-action on $\U_{t}(\g')$  by
\begin{align}
&\sigma'_i\cdot e'_j=(-1)^{(\alpha'_i,\alpha'_j)}e'_j, \quad \sigma'_i\cdot f'_j=(-1)^{-(\alpha'_i,\alpha'_j)}f'_j, \quad \sigma'_i\cdot k'_j=k'_j, \quad\text{$i\ne 0$}.
\end{align}
This induces a $\mathrm{G}'$-action on $\Dr_t(\g')$ analogous to \eqref{eq:G-act}.
Now we introduce the smash product algebras
\[
\UU_{t}(\g')=\U_{t}(\g')\sharp{K}\mathrm{G}', \quad \D_{t}(\g')=\Dr_{t}(\g')\sharp{K}\mathrm{G}'.
\]
Clearly we can extend equation \eqref{eq:JD-Dr-alg} to the algebra isomorphism
\begin{eqnarray}\label{eq:dr-iso-A}
\UU_{t}(\g')\stackrel{\sim}{\longrightarrow}\D_{t}(\g'),
\end{eqnarray}
which is the identity on ${K}\mathrm{G}'$.

\subsection{Quantum correspondences}\label{sect:QCs}

We classified the quantum correspondences in \cite[Theorem 4.9]{XZ}; 
the following ones are relevant to the present paper,
which were first established in \cite{Z2}.
\begin{lem}[\cite{XZ}, \cite{Z2}]\label{lem:correspon}
For each pair $(\g, \g')$ in Table \ref{tbl:g}, there exists an isomorphism
$
\psi: \UU_{q}(\g)\stackrel{\sim}{\longrightarrow}\UU_{t}(\g')
$
of associative algebras given by
\begin{equation}\label{eq:affine-map}
\begin{aligned}
&e_0 \mapsto \iota_{e}e'_0, \quad  \ f_0\mapsto \iota_{f} f'_0, \quad  \ k^{\pm 1}_0 \mapsto  \iota_{e} \iota_{f} k'^{\pm 1}_0,\\
& \sigma_i\mapsto \sigma'_i,\quad
e_i \mapsto \left(\prod_{k=i+1}^{m+n}\sigma'_k\right)  e'_i, \quad f_i \mapsto \left(\prod_{k=i}^{m+n}\sigma'_k\right) f'_i, \quad k_i \mapsto \sigma'_i  k'_i,
\end{aligned}
\end{equation}
for $i\neq 0$, where  $\iota_{e}, \iota_{f}\in{K}G'$ are defined by
\begin{equation*}
    \iota_{e}=\begin{cases}
1, & \g'=A_{2n}^{(2)},\\
\prod_{i=2}^n\sigma'_i, &\g'=B_{n}^{(1)},\\
\prod_{i=0}^n\sigma'_{2+2i} ,&\g'=D_{n+1}^{(2)},
           \end{cases}
\quad
  \iota_{f}=\begin{cases}
1, & \g'=A_{2n}^{(2)},\\
\prod_{i=1}^n\sigma'_i  , &\g'=B_{n}^{(1)},\\
\prod_{i=0}^n\sigma'_{1+2i},&\g'=D_{n+1}^{(2)}
           \end{cases}
\end{equation*}
with $\sigma'_j=1$ for all $j>n$.
\end{lem}

\begin{rem} 
Within the context of Hopf algebras over braided tensor categories, 
the above associative algebra isomorphism becomes a Hopf algebra 
isomorphism. The same type of Hopf superalgebra isomorphisms, referred to as quantum correspondences in \cite{XZ}, exist for a much wider range of affine Lie superalgebras \cite[Theorem 1.2]{XZ}. Some of them appear as S-dualities in string theory as discovered in \cite{MW}.
\end{rem}

Now for each $\g'$, we introduce a  surjection $o: I_n=\{1,\dots,n\}\to \{\pm 1\}$ defined by  $o(i)=(-1)^{n-i}$. We also define $c:=c(\g)$  such that
$c=1/2$ if $\g=\osp(2|2n)^{(2)}$,  and  1 otherwise.

We have the following result.
\begin{theo}\label{lem:iso}
For each pair $(\g, \g')$ in Table \ref{tbl:g},
there is an isomorphism $\varphi: \D_{q}(\g)\stackrel{\sim}{\longrightarrow}\D_{t}(\g')$
of associative algebras given by
\begin{equation}\label{eq:dr-map}
\begin{aligned}
& \ga^{1/2}\mapsto \ga'^{1/2}, \ \ \h_{i,r} \mapsto -o(i)^{r c}  \h'_{i,r},
       \ \ \qk^{\pm 1}_i \mapsto \sigma'_i  \qk'^{\pm 1}_i,
       \ \ \sigma_i\mapsto \sigma'_i,\\
&     \e_{i,r} \mapsto o(i)^{r c} \left(\prod_{k=i+1}^{m+n}\sigma'_k\right)  \x^{+}_{i,r},
\quad \f_{i,r}  \mapsto o(i)^{r c} \left(\prod_{k=i+1}^{m+n}\sigma'_k\right)  \x^{-}_{i,r}.
\end{aligned}
\end{equation}
\end{theo}
\begin{proof}
If we can show that the map $\varphi$ indeed gives rise to a homomorphism of associative algebras,  then by inspecting \eqref{eq:dr-map}, we immediately see that it is an isomorphism with the inverse map given by
\begin{equation}\label{eq:iso-inv}
\begin{aligned}
\varphi^{-1}:\quad
&  \ga'\mapsto\ga, \quad \h'_{i,r} \mapsto -o(i)^{c r}  \h_{i,r},
        \quad \qk'_i \mapsto \sigma_i  \qk_i,
        \quad \sigma'_i\mapsto \sigma_i,\\
&\      \x^{+}_{i,r} \mapsto o(i)^{c r} \left(\prod_{k=i+1}^{m+n}\sigma_k\right)  \e_{i,r},
\quad \x^{-}_{i,r}  \mapsto o(i)^{c r} \left(\prod_{k=i+1}^{m+n}\sigma_k\right)  \f_{i,r}.
\end{aligned}
\end{equation}
We prove that $\varphi$ is an algebra homomorphism by showing that
the elements $\varphi(\ef_{i,r})$, $\varphi(\h_{i,r})$, $\varphi(\qk^{\pm}_i)$, $\varphi(\ga)$, $\varphi(\sigma_i)$ in $\D_q(\g') $ satisfy the  defining relations of $\D_q(\g)$.

Let us start by verifying the first set of relations in Definition \ref{def:DR}.
Using $\x^{\pm}_{i,r}\sigma'_j=(-1)^{(\alpha'_i,\alpha'_j)}\sigma'_j\x^{\pm}_{i,r}$ and $(-1)^{(\alpha'_i,\alpha'_j)}t_i^{\pm a_{ij}}=q_i^{\pm a_{ij}}$, we immediately obtain
\[
\varphi(\qk_i) \varphi(\ef_{j,r}) \varphi(\qk_i^{-1})=q_i^{\pm a_{ij}} \varphi(\ef_{j,r}).
\]
 Since $u_{i,j,r}=o(i)^{c r}o(j)^{c r}u'_{i,j,r}$,  we have
\begin{eqnarray*}
&[\varphi(\h_{i,r}), \varphi(\ef_{j,s})]  = \dfrac{  u_{i,j,r} \varphi(\ga)^{\mp|r|/2}  }
                                                   {   r(q-q^{-1})  }
                                                   \varphi( \ef_{j,s+r}),\\
&[\varphi(\h_{i,r}),\varphi(\h_{j,s})]=\delta_{r+s,0}
                              \dfrac{ u_{i,j,r} (\varphi(\ga)^{r}-\varphi(\ga)^{-r})    }
                                                           {  r   (q-q^{-1})(q-q^{-1})  } .
\end{eqnarray*}
Let $\Phi'_{j}=\prod_{k=j}^n\sigma'_k$.  Then
\begin{eqnarray*}
\x^{+}_{n,r}\Phi'_{j}=(-1)^{\delta_{n,j}}\Phi'_{j}\x^{+}_{n,r},
&&\x^{+}_{i,r}\Phi'_{j}=(-1)^{\delta_{i,j}+\delta_{i+1,j}}\Phi'_{j}\x^{+}_{i,r}, \quad i\neq n.
\end{eqnarray*}
Using this we obtain
\begin{eqnarray*}
&&\varphi(\e_{i,r})\varphi(\f_{j,s})-(-1)^{[\e_{i,r}][\f_{j,s}]}\varphi(\f_{j,s})\varphi(\e_{i,r}) \\
&&=\de_{i,j} \dfrac{   \varphi(\ga)^{\frac{r-s}{2}} \varphi(\qk_i)\varphi(\hh^{+}_{i,r+s})
                                   -  \varphi(\ga)^{\frac{s-r}{2}}\varphi(\qk_i)^{-1}\varphi(\hh^{-}_{i,r+s})   }
                                         {  q-q^{-1}  },
\end{eqnarray*}
where we have used $\varphi(\hh^{+}_{i,r+s})=o(i)^{c  r}\hh'^{+}_{i,r+s}$ since $\varphi(\h_{i,r})= -o(i)^{c  r} \h'_{i,r}$. Now we have the obvious relations
\begin{eqnarray*}
&&[\varphi(\e_{n,r}), \varphi(\e_{j,s})]_{q^{a_{nj}}_{n}}=(-1)^{\delta_{n-1,j}}\Phi'_{j+1}
     [\x^{+}_{n,r},\x^{+}_{j,s}]_{t^{a_{nj}}_{n}},\\
&&[\varphi(\e_{i,r}), \varphi(\e_{j,s})]_{q^{a_{ij}}_{i}}
                                                      =(-1)^{\delta_{i,j}+\delta_{i-1,j}}\Phi'_{i+1}\Phi'_{j+1}
    [\x^{+}_{i,r},\x^{+}_{j,s}]_{t^{a_{ij}}_{i}}, i\neq n.
\end{eqnarray*}
Using them  together with \eqref{eq:sym-dr}, we obatin
$sym_{r,s}[\varphi(\e_{i,r+ \theta}),\varphi(\e_{j,s})]_{q^{a_{ij}}_{i}}=0$,  if $(\g,i,j)\neq (A_{2n}^{(2)}, n, n)$.  We can similarly show that $sym_{r,s}[\varphi(\f_{i,r- \theta}),\varphi(\f_{j,s})]_{q^{a_{ij}}_{i}}=0$.

The Serre relations in Definition \ref{def:DR}
can be verified in the same way.  For example, in the case $\g'=B_n^{(1)}$, we have
\[\begin{aligned}
&    sym_{r_1,r_2,r_3} \sum_{k=0}^3\begin{bmatrix} 3\\k\end{bmatrix}_{\sqrt{-1} q_{n}}
               {\varphi(\e_{n,r_1})}\dots
                          {\varphi(\e_{n,r_k})} \varphi(\e_{n-1,s}){\varphi(\e_{n,r_{k+1}})}
                  \dots {\varphi(\e_{n,r_3})}\\
&=\sigma'_n sym_{r_1,r_2,r_3}
     \sum_{k=0}^3(-1)^k\begin{bmatrix} 3\\k\end{bmatrix}_{t_n}
              {\x^{+}_{n,r_1}}\dots
                      {\x^{+}_{n,r_k}}\x^{+}_{n-1,s}{\x^{+}_{n,r_{k+1}}}
               \dots {\x^{+}_{n,r_3}}=0.
\end{aligned}
\]
We omit the proof of the other Serre relations.
\end{proof}

The map $\varphi$ becomes a Hopf superalgebra isomorphism up to picture changes and Drinfeld twists; see \cite[Theorem 1.2]{XZ} for details.

\subsection{Proof of Theorem \ref{them:DR-iso}}\label{sec:DR-osp1}
Theorem \ref{them:DR-iso} is an easy consequence of Theorem \ref{lem:iso}.
\begin{coro}
Theorem \ref{them:DR-iso} holds for each pair $(\g, \g')$ in Table \ref{tbl:g}.
\end{coro}
\begin{proof} By composing the isomorphism \eqref{eq:JD-Dr-alg} with those in Lemma \ref{lem:correspon} and Theorem  \ref{lem:iso},
we immediately  obtain the algebra isomorphism
$$\Phi=\varphi\circ\rho\circ\psi: \UU_{q}(\g)  \mapsto  \D_{q}(\g).$$
Note that $\Phi$ preserves the $\Z_2$-grading, thus is an isomorphism of superalgebras.

One can easily check that $\Phi(\U_q(\g)\otimes 1)=\Dr_q (\g)\otimes 1$. Let
$\eta: \U_{q}(\g) \longrightarrow  \UU_{q}(\g)$ be the embedding $x\mapsto x\otimes 1$,
and $\upsilon: \Dr_q (\g)\otimes 1 \longrightarrow \Dr_q (\g)$ be the natural isomorphism
$y\otimes 1\mapsto y$.  Then $\upsilon\circ\Phi\circ\eta$ is the superalgebra isomorphism $\Psi$ of Theorem \ref{them:DR-iso}.
\end{proof}

We comment on a possible alternative approach to the proof of Theorem \ref{them:DR-iso}.
For the affine Lie superalgebras in \eqref{eq:g},
the combinatorics of the affine Weyl groups of the
root systems essentially controls the structures of the affine Lie superalgebras themselves.
The corresponding quantum affine superalgebras
have enough Lusztig automorphisms, which can in principle be used to prove
Theorem \ref{them:DR-iso} by following the approach of \cite{Be}.
It will be interesting to work out the details of such a proof,
although it is expected to be much more involved than the one given here.


\section{Vertex operator  representations}\label{sect:vertex}
We construct vertex  operator representations of the quantum affine superalgebras $\Uq(\g)$ for all $\g$ in \eqref{eq:g}. These representations are level $1$
irreducible integrable highest weight representations relative to the standard triangular
decomposition \eqref{eq:triangular-1} of  $\Uq(\g)$ given below.  By level $1$ representations, we mean those with $\gamma$ acting by multiplication by $\pm q$ or  $\sqrt{-1}q$.

Our construction involves generalising to the quantum affine superalgebra context some aspects of  \cite{LP}.
The vertex operators obtained here have considerable similarities with those  \cite{Jn1, JnM} for ordinary twisted quantum affine algebras.

\subsection{Some general facts} 
We now discuss some simple facts, which  will be used to study the representation theory of the quantum affine superalgebras.
\subsubsection{Triangular decompositions}
We will need two triangular decompositions for the quantum affine superalgebra $\Uq(\g)$ for each $\g$ in \eqref{eq:g}.

The standard triangular decomposition is
\begin{eqnarray}\label{eq:triangular-1}
\begin{aligned}
&\Uq(\g)=\U_q^{(-)} \U_q^{(0)} \U_q^{(+)},  \quad {with}\\
& \U_q^{(+)} \text{  generated by  } \e_{i,0},  \e_{i,r}, \f_{i,r},   \hh_{i,r}^{\pm}, \ \text{for $r>0$, \ $1\le i\le n$}, \\
& \U_q^{(0)} \text{  generated by  } \ga_i^{\pm 1}, \ \ga^{\pm 1/2}, \ \text{for $1\le i\le n$}, \\
& \U_q^{(-)} \text{  generated by   } \f_{i,0},   \f_{i,r}, \e_{i,r},  \hh_{i,r}^{\pm}, \ \text{for $r<0$, \ $1\le i\le n$},
\end{aligned}
\end{eqnarray}
where  $\U_q^{(-)}$, $\U_q^{(0)}$ and $\U_q^{(+)}$ are all super subalgebras of $\Uq(\g)$.
In terms of the Chevalley generators in Definition \ref{defi:quantum-super},  $\U_q^{(+)}$, $\U_q^{(-)}$ and $\U_q^{(0)}$ are respectively generated by the elements $e_j$, $f_j$ and $k^{\pm 1}_j$ with $0\le j\le n$.

The other triangular decomposition is
\begin{eqnarray}\label{eq:triangular-2}
\begin{aligned}&\Uq(\g)=\U_q^- \U_q^0 \U_q^+, \quad \text{with}\\
& \U_q^+ \text{  generated by  } \e_{i,r}, \  \text{for $1\le i\le n$, \  $r\in\Z$},\\
& \U_q^0 \text{  generated by   } \hh_{i,r}^{\pm}, \ \ga^{\pm 1/2}, \  \text{for $1\le i\le n$,  \  $r\in\Z$},  \\
& \U_q^- \text{  generated by   } \f_{i,r}, \  \text{for $1\le i\le n$,  \  $r\in\Z$},
\end{aligned}
\end{eqnarray}
where $\U_q^{-}$, $\U_q^{0}$ and $\U_q^{+}$ are also super subalgebras.
The existence of this triangular decomposition is easy to see from the Drinfeld realisation, but very obscure from the point of view of Definition \ref{defi:quantum-super}.

Let $B_q:=\U_q^{(0)} \U_q^{(+)}$ or $B_q:=\U_q^0 \U_q^+$ depending on the triangular decomposition.  A vector $v_0$ in a $\Uq(\g)$-module is a highest weight vector if $\C(q) v_0$ is a $1$-dimensional $B_q$-module.  A $\Uq(\g)$-module generated by a highest weight vector is a highest weight  module with respect to the given triangular decomposition. We will study highest weight representations with respect to both triangular decompositions in later sections.

\subsubsection{Comments on spinoral type modules} \label{sect:type-1}
One can easily see that there exist the following superalgebra automorphisms of $\Uq(\g)$.
\begin{align}\label{eq:auto-1}
\iota_\varepsilon: k_i\mapsto \varepsilon_i k_i,\quad e_i\mapsto \varepsilon_i e_i,\quad f_i\mapsto f_i, \quad 0\le i\le n,
\end{align}
for any given $\varepsilon_i\in\{\pm 1\}$.
If $V$ is $\Uq(\g)$-module, we can twist it by $\iota_\varepsilon$ to obtain another module with the same underlying vector superspace but the twisted $\Uq(\g)$-action $\Uq(\g)\otimes V\longrightarrow V$ defined by
$
x\otimes v \mapsto  \iota_\varepsilon(x)v$ for all $x\in \Uq(\g) $ and $v\in V$.
If $k_i$ $(i=0,1,\dots n)$ act semi-simply on $V$, the eigenvalues of $k_i$ are multiplied by $\varepsilon_i$ in the twisted module.

Recall the notion of type-{\bf {1}} modules in the theory of ordinary quantum groups and quantum affine algebras.  For quantum supergroups and quantum affine superalgebras, a type-{\bf {1}} module over $\Uq(\g)$ is one such that the $k_i$ $(i=0,1,\dots n)$ act semi-simply with eigenvalues of the form $q_i^m$ for $m\in\Z$.
Any weight module over an ordinary quantum group or quantum affine algebra can be twisted into a type-{\bf {1}} module by analogues of the automorphisms \eqref{eq:auto-1}.  However, that is no longer true in the present context.
As we will see from Theorem \ref{theo-finite module}, some finite dimensional simple $\Uq(\g)$-modules have $k_n$-eigenvalues of the form $\pm\sqrt{-1}q^{m+1/2}$ with $m\in\Z$.  It is not possible to twist such modules into type-{\bf 1} by the automorphisms \eqref{eq:auto-1}.

For easy reference, we introduce the following definition.
\begin{defi}\label{def:type-s}
Call a $\Uq(\g)$-module type-{\bf {s}}, meaning spinoral type,  if all $k^{\pm1}_i$ act semi-simply with eigenvalues of the following form. If  $\g=\osp(1|2n)^{(1)}$ or $\Sl(1|2n)^{(2)}$, the eigenvalues of
$k_i$ for $0\le i< n$ belong to $\{q^j\mid j\in\Z\}$, and eigenvalues of $k_n$ to $\{\sqrt{-1}q^{j+1/2}\mid j\in\Z\}$.
If  $\g=\osp(2|2n)^{(2)}$, the eigenvalues of either
$k_0$, $k_n$, or both belong to $\{\sqrt{-1}q^{j+1/2}\mid j\in\Z\}$,
and the eigenvalues of the other $k_i$ to
$\{q^j\mid j\in\Z\}$.
\end{defi}

Type-{\bf s} modules exist  even for the quantum supergroup $\Uq(\osp(1|2))$ associated with $\osp(1|2)$.

\begin{example}[Type-{\bf s} representations of $\Uq(\osp(1|2))$]
The quantum supergroup $\Uq(\osp(1|2))$ is generated by $E, F$ and $K^{\pm1}$ with relations $K K^{-1}=K^{-1} K=1$ and
\[
K E K^{-1} = q E, \quad K F K^{-1} = q^{-1} F, \quad E F + F E = \frac{K-K^{-1}}{q-q^{-1}}.
\]
It has long been known that there exists an $\ell$-dimensional irreducible representation of $\Uq(\osp(1|2))$ for each positive integer $\ell$.  If $\ell$ is odd, the irreducible representation can be twisted into a type-{\bf 1} representation;  and  if $\ell$ is even, to a type-{\bf s} representation.

The smallest type-{\bf s} example is the $2$-dimensional irreducible representation, which is given by
\[
E\mapsto \begin{pmatrix}
0&\frac{\sqrt{-1}}{q^{1/2}-q^{-1/2}}\\
0&0 \end{pmatrix}, \quad
F\mapsto \begin{pmatrix} 0&0\\ 1&0 \end{pmatrix}, \quad  K\mapsto\begin{pmatrix}\sqrt{-1}q^{1/2}&0\\0&\sqrt{-1}q^{-1/2}\end{pmatrix}.
\]
\begin{rem}
The $2$-dimensional irreducible representation  of
$\Uq(\osp(1|2))$ does not have a classical limit, i.e., $q\to 1$ limit, nor do all the even dimensional irreducible representations.  This agrees with the fact that the finite dimensional irreducible representations of $\osp(1|2)$ are all odd dimensional.
\end{rem}
\end{example}

The quantum affine superalgebra $\Uq(\g)$ for all $\g$ in \eqref{eq:g} contains the quantum supergroup $\Uq(\osp(1|2))$ as a super subalgebra.  The type-{\bf s} representations of $\Uq(\g)$ restrict to type-{\bf s} representations of $\Uq(\osp(1|2))$.

\subsection{The Fock space}\label{sec:space}
Let $\ell(\alpha_i):=(\alpha_i,\alpha_i)$ for any simple root $\alpha_i$.
For convenience,
we choose the normalisation for the bilinear form so that $\ell(\alpha_n)=2$ if $\g=\osp(2|2n)^{(2)}$,  and $\ell(\alpha_n)=1$ otherwise. Let $\wp=(-1)^{1/\ell(\alpha_n)}q$, and take $\wp^{1/2}=(-1)^{\frac{1}{2\ell(\alpha_n)}}q^{1/2}$.

Hereafter we will always consider $\Uq(\g)$ in the Drinfeld realisation  given in Definition \ref{def:DR} and Lemma \ref{lem:dr-f}.
Denote by $\Uq(\widetilde{\eta})$
the subalgebra of $\Uq(\g)$ generated by the elements  $\gamma^{1/2}$, $\qk_i$ and $\h_{i,r}$ ($r\in\Z\backslash{\{0\}}$, $1\le i\le n$),
and by $\U_q(\eta)$ that generated $\gamma^{1/2}$ and $\h_{i,r}$  ($r\in\Z\backslash{\{0\}}$, $1\le i\le n$).
Let $S(\eta^{-})$ be the symmetric algebra generated by $\h_{i,r}$ for $r\in\Z_{<0}$ and $1\le i\le n$.  Let $H_i(s)$ ($s\in\Z\backslash\{0\}$, $1\le i\le n$) be the linear operators acting on $S(\eta^{-})$ such that
\begin{eqnarray}\label{eq:vo-H}
\begin{aligned}
&\HH_i(-s)=\text{derivation defined by}\\
&\phantom{HH_i(-s)} \HH_i(-s)(\h_{j,r})=\delta_{r,s} \dfrac{ u_{i,j,-s} (\wp^{s}-\wp^{-s})  }
                                                           { s   (q_i-q_i^{-1})(q_j-q_j^{-1})  }, \\
&\HH_i(s)=\text{multiplication by $\h_{i,s}$}, \qquad \forall r, s\in\Z_{<0},
\end{aligned}
\end{eqnarray}
where $u_{i,j,-s}$ is defined by \eqref{eq:u-def}. Then
\begin{align}\label{eq:vo-hh}
[\HH_{i}(r),\HH_{j}(s)]=\delta_{r+s,0}\dfrac{u_{i,j,r} (\wp^{r}-\wp^{-r})  }
                                  { r   (q_i-q_i^{-1})(q_j-q_j^{-1})  },
                                 \quad \forall r,s\in\Z\backslash\{0\}.
\end{align}
The algebra  $\U_q(\eta)$ has the canonical irreducible representation on $S(\eta^{-})$ given by
\[
\begin{aligned}
\ga \mapsto \wp, \quad \h_{i,s} \mapsto \HH_i(s), \quad \forall   s\in\Z\backslash\{0\}.
\end{aligned}
\]

Let $\dot{\g}\subset\g$ be the regular simple Lie sub-superalgebra with the Dynkin diagram obtained from the Dynkin diagram of $\g$ by removing the node corresponding to $\alpha_0$. Then $\dot{\g}=\osp(1|2n)$ in all three cases of $\g$. Let $\Q$ be the root lattice of $\dot{\g}$ with the bilinear form inherited from that of $\g$. We regard $\Q$ as a multiplicative group consisting of elements of the form $e^\alpha$ with $\alpha\in \Q$. Let $\C[\Q]$ be the group algebra of $\Q$. Given any variable $z$ and any root $\alpha$,  we define a linear operator on $\C[\Q]$ by
\begin{eqnarray}\label{eq:operator-on-CQ}
 z^{\alpha}. e^{\beta}=z^{(\alpha,\beta)}e^{\beta}.
\end{eqnarray}
We also define the linear operator $\sigma_i$ on $\C[\Q]$ for all $i=1, 2, \dots, n$  by
\[\begin{aligned}
\sigma_i. e^{\beta}=(-1)^{(\alpha_i,\beta)}e^{\beta}.
\end{aligned}\]
Write $\Phi_i=\prod_{k=i}^{n}\sigma_k$ for $1\le i\le n$ and $\Phi_i=1$ for $i>n$. It is easy to check that $\Phi_i.e^{\pm\alpha_j}=(-1)^{\delta_{i,j}+\delta_{i+1,j}}e^{\pm\alpha_j}$ for $1\le i,j \le n$ and $\Phi_i^2=1$.

We also need some basic knowledge of the $q$-deformed Clifford algebra $\cq$,
which is generated by $\ka(r),\ka(s)$ ($r,s\in\Z +\frac{1}{2}$) with relations
\begin{equation}\label{eq:kk}
\ka(r)\ka(s)+\ka(s)\ka(r)=\delta_{r,-s}(q^r+q^{s}), \quad \forall r, s.
\end{equation}
We use $\Lambda(\cq^{-})$ to denote the exterior algebra generated by $\ka(r)$ for $r<0$,  and denote by $\Lambda(\cq^{-})_0$ (resp. $\Lambda(\cq^{-})_1$) the subspace of even (resp. odd) degree, where $\ka(r)$ ($r<0$) are regarded as having degree $1$.  Define the linear operators $P (s)$ on $\Lambda(\cq^{-})$ such that for any $\psi, \phi\in \Lambda(\cq^{-})$,
\[
\begin{aligned}
&P(s)\cdot\psi = \ka(s)\psi, \quad P (-s)\cdot\ka(r)=\delta_{r,s}(q^r+q^{-r}), \quad P (-s)\cdot 1=0, \\
&P(-s)\cdot(\psi\phi)= P(-s)\cdot(\psi) \phi + (-1)^{deg(\psi)}\psi P(-s)\cdot(\phi), \quad \forall r, s<0.
\end{aligned}
\]
Then $\cq$ acts on $\Lambda(\cq^{-})$ by
$\ka(r)\mapsto K(r)$ for all $r\in\Z+\frac{1}{2}$.
Let
\begin{align}\label{eq:v}
W=\left\{
\begin{aligned}
&\C[\Q],  \quad
                      \g=\osp(1|2n)^{(1)},\osp(2|2n)^{(2)};\\
&\C[\Q_0]\otimes \Lambda(\mathcal{C}_{\wp}^{-})_0  \oplus \C[\Q_0]e^{\lambda_1}\otimes \Lambda(\mathcal{C}_{\wp}^{-})_1,   \quad
                      \g=\Sl(1|2n)^{(2)},
\end{aligned} \right.
\end{align}
where $\Q_0$ is the lattice spanned by the set of roots with squared length 2 and $\lambda_1=\alpha_1+\alpha_2+\dots+\alpha_n$.
Now we construct the vector space
$
V=S(\eta^{-})\otimes W.
$


\subsection{Construction of the vacuum representations}\label{sec-vo-0}
We start by defining
\begin{align*}
&P(z)=\sum_{s\in\Z +1/2} P(s)z^{-s}, \\
&T^{+}_i(z)=\begin{cases}
e^{\alpha_i} \Phi_i z^{\alpha_i+\ell(\alpha_i)/2}, \quad \text{if  }  \g=\osp(1|2n)^{(1)}, \osp(2|2n)^{(2)};\\
e^{\alpha_i}  \Phi_i  z^{\alpha_i+\ell(\alpha_i)/2} P(z), \quad \text{if  }  \g=\Sl(1|2n)^{(2)},
\end{cases}\\
&T^{-}_i(z)=\begin{cases}
e^{-\alpha_i} \Phi_{i+1} z^{-\alpha_i+\ell(\alpha_i)/2}, \quad \text{if  } \g=\osp(1|2n)^{(1)}, \osp(2|2n)^{(2)};\\
e^{-\alpha_i}  \Phi_{i+1}  z^{-\alpha_i+\ell(\alpha_i)/2}(- P(z)), \quad \text{if  }  \g=\Sl(1|2n)^{(2)},
\end{cases}
\end{align*}
and introducing the following formal distributions:
\begin{align*}
&E^{\pm}_{i}(z) =\xp\left(
        \pm\sum_{k=1}^{\infty}\frac{\wp^{\mp k/2}}{ \{k\}_{q_i} }\HH_i(-k)z^k
                                  \right),\\
&F^{\pm}_{i}(z) =\xp\left(
       \mp\sum_{k=1}^{\infty}\frac{\wp^{\mp k/2}}{ \{k\}_{q_i}}\HH_i(k)z^{-k}
                                 \right),
\end{align*}
where $\{k\}_{q_i}=[k]_{\wp}\cdot \frac{\wp-\wp^{-1}}{q_i-q_i^{-1}}=\frac{\wp^k-\wp^{-k}}{q_i-q_i^{-1}}$.
Using them, we define linear operators $\X^{\pm}_j(k)$ ($1\le j\le n$,   $k\in\Z$) on the vector space $V$ by
\begin{align}\label{eq:vo}
\X^{\pm}_{i}(z)=E^{\pm}_{i}(z)F^{\pm}_{i}(z)T^{\pm}_i(z), \quad i=1, 2, \dots, n,
\end{align}
where
\[
\begin{aligned}
\X^{\pm}_i(z)&=\sum_{k\in\Z} \X^{\pm}_i(k) z^{-k}, \ \qquad \text{for all $i\ne n$,}\\
\X^{\pm}_n(z)&=\sum_{k\in\Z} \X^{\pm}_n(k) z^{-k}, \ \qquad \text{if  $\g\ne\osp(1|2n)^{(1)}$,}\\
\X^{\pm}_n(z)&=\sum_{k\in\Z} \X^{\pm}_n(k) z^{-k+1/2},  \quad \text{ if $\g=\osp(1|2n)^{(1)}$}.
\end{aligned}
\]

We have the following result.
\begin{theo}\label{them:v.o}
The quantum affine superalgebra $\Uq(\g)$ acts irreducibly
on the vector space  $V$, with the action defined by
\begin{eqnarray}\label{eq:action}
\begin{aligned}
&\ga^{1/2}\mapsto \wp^{1/2},\ \ \qk_i^{1/2}\mapsto (\varpi_i\sigma_i \wp^{\alpha_i})^{1/2},
\ \  \h_{i,s}\mapsto \HH_i(s),\\
&\e_{i,k}\mapsto \X^{+}_i(k),\ \ \f_{i,k}\mapsto \varrho_i\X^{-}_i(k),\\
&\forall i=1, 2, \dots, n, \ \ s\in\Z\backslash\{0\}, \ \ k\in\Z,
\end{aligned}
\end{eqnarray}
where
\[\begin{aligned}
& \varpi_i=\begin{cases}
  \wp^{-1/2}, &\text{if \  } \g=\osp(1|2n)^{(1)}, \ i= n;\\
   1,               & \text{otherwise};
             \end{cases} \\
&\varrho_i=\begin{cases}
 -2^{-1}\{\ell(\alpha_i)/2\}_{q_i}, &\text{if \  } \g=\osp(2|2n)^{(2)}, \ i\neq n;\\
 -\{\ell(\alpha_i)/2\}_{q_i},           & \text{otherwise}.
                \end{cases}
\end{aligned}\]
\end{theo}

\begin{proof}
The irreducibility of $V$ as a $\Uq(\g)$-module follows from the fact that the $\Uq(\eta)$-module $S(\eta^{-})$ and $\Uq(\widetilde{\eta})$-module  $W$ are both irreducible. This was proved in \cite{Jn96}.
Thus the proof of the theorem essentially boils down to verifying that the operators
$ \HH_i(k)$ and $\X^{\pm}_i(k)$ satisfy the commutation relations of
$\h_{i,k}$ and $\ef_{i,k}$. We show this by using  calculus of
 formal distributions.

Consider the vertex operators \eqref{eq:vo} in the case of  $\g=\osp(1|2n)^{(1)}$.
We claim that they satisfy the following relation (cf. \eqref{eq:xx-f}):
\begin{eqnarray}\label{eq:vo-XX}
\begin{aligned}
&[\X^{+}_{i}(z), \X^{-}_{j}(w)]
=\frac{\delta_{i j}\,\rho_{z,w}\,\varrho^{-1}_i}{q_i-q_i^{-1}}
     \left\{\sigma_i \varpi_i \wp^{\alpha_i} \widetilde{V}_i^+(z)
                  \delta\left(\wp^{-1}\frac{z}{w}\right)
      \right.\\
&\hspace{28mm} \left.
     -\sigma_i \varpi_i \wp^{-\alpha_i}
\widetilde{V}_i^-(z)
       \delta\left(\wp\frac{z}{w}\right) \right\},
\end{aligned}
\end{eqnarray}
where
\begin{eqnarray}\label{eq:widetildeV}
\begin{aligned}
&\widetilde{V}_i^+(z)=\xp\left(
            \sum_{k=1}^{\infty}   (q_i-q_i^{-1})    \HH_i(k)(z\wp^{-1/2})^{-k}     \right), \\
&\widetilde{V}_i^-(z)=\xp\left( \sum_{k=1}^{\infty}(q_i^{-1}-q_i)\HH_i(-k)(z\wp^{1/2})^{k}\right).
\end{aligned}
\end{eqnarray}

If $(\alpha_i,\alpha_j)=0$ (necessarily $i\ne j$), the claim is clear.

If  $(\alpha_i,\alpha_j)\neq 0$, there are three possibilities:
$(\alpha_i,\alpha_j)=-1$ with $i\ne j$,  and $(\alpha_i,\alpha_j)=1$ or $2$ with $i= j$.

Define normal ordering as usual by placing $\HH_i(-k)$ with $k>0$ on the left of
$\HH_j(k)$, $\exp^\alpha$ on the left of $z^\beta$, and $K(-s)$ with $s>0$ on the left of $K(s)$, where for the $K(r)$'s an order change procures a sign.
Let
\begin{align*}
:T^{+}_i(z)T^{-}_j(w):&=e^{ \alpha_i-\alpha_j} \Phi_i\Phi_{j+1}  z^{ \alpha_i}w^{-\alpha_j},
\end{align*}
then we have the following relations:  if $(\alpha_i,\alpha_j)\neq 1$,
\begin{align*}
:T_i^+(z)T_j^-(w):&=(-1)^{\delta_{i-1,j}+\delta_{i,j}}  T_i^+(z)T_j^-(w)  z^{(\alpha_i,\alpha_j)}   z^{-\ell(\alpha_i)/2}   w^{-\ell(\alpha_j)/2}\\
&=(-1)^{\delta_{i-1,j}+\delta_{i,j}}  T_j^-(w)T_i^+(z) w^{(\alpha_i,\alpha_j)} z^{-\ell(\alpha_i)/2}  w^{-\ell(\alpha_j)/2};
\end{align*}
if $(\alpha_i,\alpha_j)= 1$,
\begin{align*}
:T_i^+(z)T_j^-(w):&=-T_i^+(z)T_j^-(w)  z^{(\alpha_i,\alpha_j)} z^{-\ell(\alpha_i)/2}  w^{-\ell(\alpha_j)/2}\\
&=T_j^-(w)T_i^+(z) w^{(\alpha_i,\alpha_j)} z^{-\ell(\alpha_i)/2}  w^{-\ell(\alpha_j)/2}.
\end{align*}
Also
\begin{eqnarray}\label{eq:normal order-X}
\begin{aligned}
:\X^{+}_{i}(z)\X^{-}_{j}(w):=E^{+}_{i}(z)E^{\pm}_j(w)F^{+}_{i}(z)F^{-}_{j}(w):T^{+}_i(z)T^{-}_j(w):.
\end{aligned}
\end{eqnarray}
Thus $\X^{-}_{j}(w)\X^{+}_{i}(z)$ can be expressed as
\begin{align*}
:\X^{+}_{i}(z)\X^{-}_{j}(w):  \xp \left( \sum_{k=1}^{\infty} \frac{u_{i,j,k}}{k(\wp^k-\wp^{-k})}z^{-k}w^k \right)z^{(\alpha_i,-\alpha_j)}z^{\ell(\alpha_i)/2}w^{\ell(\alpha_j)/2},
\end{align*}
where we have used the Baker-Campbell-Hausdorff formula.

Let  $\delta_1(x)=\sum_{n\le 0}(\wp^{-n}-\wp^{n})x^n$. Then direct computation shows that
$\X^{+}_{i}(z)\X^{-}_{j}(w)$ can be expressed as
\begin{align*}
  & :\X^{+}_{i}(z)\X^{-}_{j}(w): (-1)^{\delta_{i-1,j}}  (z+w)\, z^{\ell(\alpha_i)/2}w^{\ell(\alpha_j)/2} , \quad \text{if \ }  (\alpha_i,\alpha_j)=-1,\\
  & :\X^{+}_{i}(z)\X^{-}_{j}(w): \frac{1}{\wp-\wp^{-1}}\delta_1(z/w),  \quad \text{if \ }  (\alpha_i,\alpha_j)=2,\\
  & :\X^{+}_{i}(z)\X^{-}_{j}(w): \frac{1}{\wp-\wp^{-1}} \delta_1(z/w)\, (z+w)(zw)^{-1/2}, \quad \text{if \ }  (\alpha_i,\alpha_j)=1,
\end{align*}
where we have used the formula $\text{ln}(1-x)=-\sum_{n=1}^{\infty}\frac{x^n}{n}$.
Note that $z^{\pm 1/2}$ and $w^{\pm 1/2}$ may appear in $\X^{+}_{i}(z)\X^{-}_{j}(w)$.
A similar computation shows that
\begin{itemize}
\item if $(\alpha_i,\alpha_j)=-1$,
\begin{align*}
 \X^{-}_{j}(w)\X^{+}_{i}(z)=:\X^{+}_{i}(z)\X^{-}_{j}(w): (-1)^{\delta_{i-1,j}} (z+w)\, z^{\ell(\alpha_i)/2}w^{\ell(\alpha_j)/2},
\end{align*}

\item if $ (\alpha_i,\alpha_j)=2,$
\begin{align*}
\X^{-}_{j}(w)\X^{+}_{i}(z)=:\X^{+}_{i}(z)\X^{-}_{j}(w): \frac{1}{\wp-\wp^{-1}}\,\delta_1(w/z),
\end{align*}

\item if $(\alpha_i,\alpha_j)=1, $
\begin{align*}
&&\X^{-}_{j}(w)\X^{+}_{i}(z)=:\X^{+}_{i}(z)\X^{-}_{j}(w): \frac{1}{\wp^{-1}-\wp}\, \delta_1(w/z)\, (z+w)(zw)^{-1/2}.
\end{align*}
\end{itemize}
Using these we obtain
\begin{align*}
&[\X^{+}_{i}(z),\X^{-}_{j}(w)]=\X^{+}_{i}(z) \X^{-}_{j}(w)-(-1)^{[\alpha_i][\alpha_j]}\X^{-}_{j}(w) \X^{+}_{i}(z)\\
=&\begin{cases}
:\X^{+}_{i}(z)\X^{-}_{j}(w):\frac{(z+w)(zw)^{-1/2}}{\wp-\wp^{-1}}\left(\delta(\wp^{-1}z/w)-\delta(\wp z/w)\right),
                     &(\alpha_i,\alpha_j)=1,\\
:\X^{+}_{i}(z)\X^{-}_{j}(w):\frac{1}{\wp-\wp^{-1}}\left(\delta(\wp^{-1}z/w)-\delta(\wp z/w)\right),
                     &(\alpha_i,\alpha_j)=2,\\
0,                  &(\alpha_i,\alpha_j)=-1,
\end{cases}
\end{align*}
where $[\alpha_i]=0$ if $\alpha_i$ is an even root, and 1 otherwise. This in particular shows that  \eqref{eq:vo-XX} holds  for all $i\ne j$.

In the cases with $i=j$, by using  $f(z,w)\delta(\frac{w}{z})=f(z,z)\delta(\frac{w}{z})$, we obtain
\begin{align*}
&:\X^{+}_{i}(z)\X^{-}_{j}(w): \delta\left(\wp^{-1}\frac{z}{w}\right)
=-\sigma_i \wp^{\alpha_i}
        \widetilde{V}_i^+(z)
         \delta\left(\wp^{-1}\frac{z}{w}\right), \\
&:\X^{+}_{i}(z)\X^{-}_{j}(w): \delta\left(\wp \frac{z}{w}\right)
=-\sigma_i \wp^{-\alpha_i}
\widetilde{V}_i^-(z)
       \delta\left(\wp\frac{z}{w}\right),
\end{align*}
where $\widetilde{V}_i^+(z)$ and $\widetilde{V}_i^-(z)$ are defined by \eqref{eq:widetildeV}. Note that
\[
\begin{aligned}
&(z+w)(zw)^{-1/2}\delta\left(\wp^{\pm1}\frac{z}{w}\right)=(z/w)^{1/2}(1+\wp^{\pm1})\delta\left(\wp^{\pm1}\frac{z}{w}\right).
\end{aligned}
\]
These formulae immediately lead to \eqref{eq:vo-XX}.

To consider the Serre relations, we take as an example the relation \eqref{eq:xrs-xsr-f} when $(\alpha_i,\alpha_j)=-1$. In this case, \eqref{eq:xrs-xsr-f} is equivalent to
\begin{align*}
   (z-q^{-1}w)\e_i(z)\e_j(w)
=(q^{-1}z-w)\e_j(w)\e_i(z).
\end{align*}
Thus, we need to show
\begin{align}\label{eq:XX}
 (z-q^{-1}w)\X^{+}_{i}(z)\X^{+}_{j}(w)
=(q^{-1}z-w)\X^{+}_{j}(w)\X^{+}_{i}(z).
\end{align}

Let $:T^{+}_i(z)T^{+}_j(w):=e^{ \alpha_i+\alpha_j} \Phi_i\Phi_{j} z^{ \alpha_i}w^{\pm\alpha_j},$ and
\[\begin{aligned}
:\X^{+}_{i}(z)\X^{+}_{j}(w):=E^{+}_{i}(z)E^{+}_j(w)F^{+}_{i}(z)F^{+}_{j}(w):T^{+}_i(z)T^{+}_j(w):.
\end{aligned}\]
By \eqref{eq:normal order-X}, $\X^{+}_{i}(z)\X^{+}_{j}(w)$ is equal to
\begin{align*}
&:\X^{+}_{i}(z)\X^{+}_{j}(w):
\xp\left[    -\sum_{k=1}^{\infty} \frac {   \wp^{-k}} {  \{k\}_{q_i}\{k\}_{q_j}  }  [\HH_i(k),\HH_j(-k)]
               \left( \frac{w}{z} \right)^k     \right] z^{-1} z^{\ell(\alpha_i)} w^{\ell(\alpha_j)},
               \end{align*}
which can be simplified to
 $:\X^{+}_{i}(z)\X^{+}_{j}(w): \left(1-q^{-1}\frac{w}{z}\right)^{-1} z^{-1} z^{\ell(\alpha_i)} w^{\ell(\alpha_j)}.
$
Thus
\[
\X^{+}_{i}(z)\X^{+}_{j}(w) =:\X^{+}_{i}(z)\X^{+}_{j}(w): (-1)^{\delta_{i-1,j}} \left(z-q^{-1}w\right)^{-1}  z^{\ell(\alpha_i)} w^{\ell(\alpha_j)}.
\]
Similarly we can show that
\begin{align*}
\X^{+}_{j}(w)\X^{+}_{i}(z)=:\X^{+}_{i}(z)\X^{+}_{j}(w): (-1)^{\delta_{i,j-1}} \left(w-q^{-1}z\right)^{-1} z^{\ell(\alpha_i)} w^{\ell(\alpha_j)}.
\end{align*}
Note that $i=j-1$ or $j+1$ in this case. Then two relations above immediately imply \eqref{eq:XX}.

Similar computation proves the theorem for the other $\g$.
\end{proof}

\begin{rem}
The representations in Theorem \ref{them:v.o} are not of type-{\bf 1}.
Note in particular that $\gamma$ acts by $\wp$. However, we can twist them into type-{\bf 1} or type-{\bf s} representations (see Definition \ref{def:type-s}) by the automorphisms  \eqref{eq:auto-1}.
\end{rem}


\subsection{Construction of the other level $1$ irreducible representations}\label{sect:other-level-1}
We now consider the vertex operator construction for the other level $1$ irreducible integrable highest weight representations with respect to the standard triangular
decomposition \eqref{eq:triangular-1}.
Observe that for  $\Uq(\osp(1|2n)^{(1)})$, the vacuum representation is the only such representation. Thus we will consider $\Uq(\g)$ for
$\g=\Sl(1|2n)^{(2)}$ and $\osp(2|2n)^{(2)}$ only.
We will only state the main results;  their proofs are quite similar to those in Section \ref{sec-vo-0}.

We maintain the notation of Section \ref{sec-vo-0}.

\subsubsection{The case of  $\Uq(\Sl(1|2n)^{(2)})$ }
There is only one irreducible integrable highest weight representation at level 1 beside the vacuum representation. It can be constructed as follows.

Recall the definition of $W$ in  \eqref{eq:v}.  Let $\lambda_n$ be the fundamental weight of $\dot{\g}$ corresponding to $\alpha_n$, and consider the subset
$\lambda_n+\Q$ of the weight lattice of $\dot{\g}$.  The linear operators $z^{\alpha}$ defined by \eqref{eq:operator-on-CQ} act on the group algebra of the weight lattice of
$\dot{\g}$ in the obvious way.
Denote $W_n=e^{\lambda_n}\C[\Q]$ and $V_n=S(\eta^{-})\otimes W_n$. Then $V_n$ is the level $1$ simple $\Uq(\Sl(1|2n)^{(2)})$-module with the action give by \eqref{eq:action} in terms of vertex operators. The highest weight vector is $1\otimes e^{\lambda_n}$.

\subsubsection{The case of $\Uq(\osp(2|2n)^{(2)})$ }
There are another two simple integrable highest weight modules at level $1$,  respectively associated with the fundamental weights $\lambda_1$ and $\lambda_n$ of $\dot{\g}$. Here $\lambda_1$ and $\lambda_n$ correspond to $\alpha_1$ and $\alpha_n$ respectively.  To construct the representations, we need the following q-deformed Clifford algebra $\cqq$, which is generated by $\ta(r),\ta(s)$ ($r,s\in\Z$) with relations
\begin{equation}\label{eq:tt}
\ta(r)\ta(s)+\ta(s)\ta(r)=\delta_{r,-s}(q^r+q^{s}), \quad \forall r, s.
\end{equation}
These are q-deformed Ramond fermionic operators.
Similar to Section \ref{sec:space}, we define linear operators $T (s)$ acting on $\Lambda(\cqq^{-})$ such that for any $\psi, \phi\in \Lambda(\cqq^{-})$,
\[
\begin{aligned}
&T(s)\cdot\psi = \ta(s)\psi, \quad T (-s)\cdot\ka(r)=\delta_{r,s}(q^r+q^{-r}), \quad T (-s)\cdot 1=0, \\
&T(-s)\cdot(\psi\phi)= T(-s)\cdot(\psi) \phi + (-1)^{deg(\psi)}\psi T(-s)\cdot(\phi), \quad \forall r, s<0,
\end{aligned}
\]
and $T(0)$ acts as the identity.

We replace $P(z)$ in Section \ref{sec-vo-0} by
$
P(z)=\sum_{s\in\Z} T(s)z^{-s}
$
and use it in \eqref{eq:vo} to obtain the corresponding vertex operators.
Now define
\[\begin{aligned}
&V^{(1)}=S(\eta^{-})\otimes W^{(1)} \quad \text{and} \quad V^{(n)}=S(\eta^{-})\otimes W^{(n)} \ \text{ with} \\
&W^{(1)}=e^{\lambda_1}\C[\Q_0]\otimes \Lambda(\mathcal{C}_{\wp}^{-})_0 \oplus \C[\Q_0]\otimes \Lambda(\mathcal{C}_{\wp}^{-})_1,\\
&W^{(n)}=e^{\lambda_n}\C[\Q]\otimes \Lambda(\mathfrak{C}_{\wp}^{-}).
\end{aligned}\]
Then $V^{(1)}$ and $V^{(n)}$ are the simple $\Uq(\osp(2|2n)^{(2)})$-modules at level $1$  with the actions formally given by \eqref{eq:action} but in terms of the new vertex operators.  The highest weight vectors are $1\otimes e^{\lambda_1}\otimes 1$ and $1\otimes e^{\lambda_n}\otimes 1$ respectively.


\subsection{Another construction of vacuum representations}\label{sect:other-vacuum}
For the quantum affine superalgebras
$\Uq(\osp(1|2n)^{(1)})$ and $\Uq(\Sl(1|2n)^{(2)})$, it is possible to modify the
vertex operators of the vacuum representations to make $\gamma$ act by $q$,
and this is what we will do in this section. The modified vertex operator representation of $\Uq(\Sl(1|2n)^{(2)})$ given here is of type-{\bf 1}.

For both affine superalgebras,  we choose in this section the normalisation for the bilinear form on the weight space so that $(\alpha_n,\alpha_n)=1$.

Recall the definitions of  $ \U_q(\eta)$ and $S(\eta^{-})$ in section \ref{sec:space}.  Let us now define new linear operators acting on $S(\eta^{-})$, denoted by
$H_i^q(s)$ with $s\in\Z\backslash\{0\}$, $1\le i\le n$,  as follows.
\begin{eqnarray}\label{eq:vo-Hp}
\begin{aligned}
&\HH_i^q(-s)=\text{derivation defined by}\\
&\phantom{HH_iq(-s)} \HH_i^q(-s)(\h_{j,r})=\delta_{r,s} \dfrac{ u_{i,j,-s} (q^{s}-q^{-s})  }
                                                           { s   (q_i-q_i^{-1})(q_j-q_j^{-1})  }, \\
&\HH_i^q(s)=\text{multiplication by $\h_{i,s}$}, \qquad \forall r, s\in\Z_{<0},
\end{aligned}
\end{eqnarray}
where $u_{i,j,-s}$ is defined by \eqref{eq:u-def}.  This differs from
\eqref{eq:vo-H} in that $\wp$  is replaced by $q$. Now we have
\begin{align}\label{eq:vo-hh-q}
[\HH^q_{i}(r),\HH^q_{j}(s)]=\delta_{r+s,0}\dfrac{u_{i,j,r} (q^{r}-q^{-r})  }
                                  { r   (q_i-q_i^{-1})(q_j-q_j^{-1})  },
                                 \quad \forall r,s\in\Z\backslash\{0\},
\end{align}
and we obtain the following irreducible $\U_q(\eta)$-representation on $S(\eta^{-})$
\[
\begin{aligned}
\ga \mapsto q, \quad \h_{i,s} \mapsto \HH_i(s), \quad \forall   s\in\Z\backslash\{0\}.
\end{aligned}
\]

Define the $2$-cocycle $C:\Q \times \Q \to \{\pm 1\}$, satisfying
\[
\begin{aligned}
         C(\alpha+\beta,\ga )=C(\alpha,\ga)C(\beta,\ga),
\quad C(\alpha,\beta+\ga)=C(\alpha,\beta)C(\alpha,\ga),
\quad \forall \alpha, \beta, \ga,
\end{aligned}
\]
such that $ C(0, \beta)=C(\alpha, 0)=1$, and for  any simple roots $\alpha_i$ and $\alpha_j$,
\[\begin{aligned}
C(\alpha_i,\alpha_j)=\begin{cases}
(-1)^{(\alpha_i,\alpha_j)+(\alpha_i,\alpha_i)(\alpha_j,\alpha_j)}, & i\le j,\\
1,& i>j.
\end{cases}
\end{aligned}\]
Obviously,  $\Q$ has a unique central extension $\hQ$,
\[\begin{aligned}
1\rightarrow  \Z_2 \rightarrow   \hQ  \rightarrow \Q\rightarrow 1
\end{aligned}\]
defined in the following way.  We regard $\hQ$ as a multiplicative group consisting of elements $\pm e^\alpha$ with $\alpha\in \Q$. Then $(-1)^a e^\alpha (-1)^b e^\beta=(-1)^{a+b}C(\alpha,\beta)e^{\alpha+\beta}$, where $a, b\in\{0, 1\}$ and $\alpha, \beta\in \Q$.
Let $\C[\hQ]$ be the group algebra of  $\hQ$, and let $J$ be the two-sided ideal generated by $e^\alpha +(-e^\alpha)$ for all $\alpha$. Denote the quotient $\C[\hQ]/J$ by $\C[\Q]$.
Now $\pm e^\alpha\in\C[\hQ]$ are natural linear operators acting on $\C[\Q]$. The linear operators $z^{\alpha}$ are the same as in section \ref{sec:space}.

Let $V=S(\eta^{-})\otimes W$, where $W$ is defined in \eqref{eq:v}.

For all $i=1, \dots, n$, let
\begin{align*}
\widetilde{T}^{\pm}_i(z)=\begin{cases}
e^{\pm\alpha_i} z^{\pm\alpha_i+\ell(\alpha_i)/2}, & \g=\osp(1|2n)^{(1)};\\
e^{\pm\alpha_i} z^{\pm\alpha_i+\ell(\alpha_i)/2}(\pm K(z)), & \g=\Sl(1|2n)^{(2)}.
\end{cases}
\end{align*}
Define
\[
\begin{aligned}
&\widetilde{E}^{\pm}_{i}(z)=\xp\left(
        \pm\sum_{k=1}^{\infty}\frac{q^{\mp k/2}}{ [k]_{q_i} }\HH^q_i(-k)z^k
                                  \right), \\
&\widetilde{F}^{\pm}_{i}(z)=\xp\left(
       \mp\sum_{k=1}^{\infty}\frac{q^{\mp k/2}}{ [k]_{q_i}}\HH^q_i(k)z^{-k}
                                 \right), \quad \text{for $i\ne n$}; \\
&\widetilde{E}^{\pm}_n(z)=\xp\left(
        \pm\sum_{k=1}^{\infty}\frac{q^{\mp k/2}}{ [2k]_{q_n} }\HH^q_n(-k)z^k
                                  \right), \\
&\widetilde{F}^{\pm}_n(z)=\xp\left(
       \mp\sum_{k=1}^{\infty}\frac{q^{\mp k/2}}{ [2k]_{q_n}}\HH^q_n(k)z^{-k}
                                 \right),
\end{aligned}
\]
and finally set
\begin{align}\label{eq:vo-p}
\widetilde{\X}^{\pm}_{i}(z) &=\widetilde{E}^{\pm}_{i}(z) \widetilde{F}^{\pm}_{i}(z)
                            \widetilde{T}^{\pm}_i(z), \quad \forall i.
\end{align}
%
Similar arguments as those in the proof of Theorem \ref{them:v.o} can prove the following result.
\begin{theo}\label{them:v.o-q}
Let $\g$ be $\osp(1|2n)^{(1)}$ or $\Sl(1|2n)^{(2)}$. Then the quantum affine superalgebra $\Uq(\g)$ acts irreducibly
on the vector space  $V$ with the action defined by
\begin{eqnarray}\label{eq:action-1}
\begin{aligned}
&\ga^{1/2}\mapsto q^{1/2},\ \  \qk_i^{ 1/2}\mapsto (\varpi_i q^{\alpha_i})^{1/2}, \ \  \h_{i,s}\mapsto \HH^q_i(s),\\
&  \e_{i,k}\mapsto \widetilde{\X}^{+}_i(k),\ \ \f_{i,k}\mapsto \vartheta_i \widetilde{\X}^{-}_i(k),\\
&\forall i=1, \dots, n, \ \ s\in\Z\backslash\{0\}, \ \ k\in\Z,
\end{aligned}
\end{eqnarray}
where $\varpi_i=\sqrt{-q}$ if $\g=\osp(1|2n)^{(1)}$ and $i=n$,  and $\varpi_i=1$ otherwise; $\vartheta_i=\frac{q_i+q_i^{-1}}{q_i-q_i^{-1}}\varpi_i^{-1}$ if $i=n$, and $\vartheta_i=1$ otherwise.
\end{theo}

\begin{rem}\label{rem:algebra automorphism} The vertex operator representation in Theorem \ref{them:v.o-q} can be changed to that in Theorem \ref{them:v.o} by the following automorphism of $\Uq(\g)$:
\[\begin{aligned}
&\ga\mapsto -\ga,\quad \e_{i,k}\mapsto \e_{i,k},\quad \f_{i,k}\mapsto (-1)^k\f_{i,k},\\
        &\qk_i^{\pm1/2}\mapsto \qk_i^{\pm1/2},\quad \h_{i,k}\mapsto (-1)^{|k|/2}\h_{i,k}, \quad \hh^{\pm}_{i,k}\mapsto (-1)^{\pm k/2}\hh^{\pm}_{i,k}.
\end{aligned}
\]
\end{rem}


\section{Finite dimensional irreducible representations}
In this section, we classify the finite dimensional irreducible representations of the quantum affine superalgebra $\Uq(\g)$ for the affine Lie superalgebras $\g$ in \eqref{eq:g}. We always assume that $\Uq(\g)$-modules are $\Z_2$-graded (cf. Remark \ref{rem:gradings}).

We choose the normalisation for the bilinear form on the weight space of $\g$ so that
$
(\alpha_n,\alpha_n)=1.
$



\subsection{Classification of finite dimensional simple modules}
We fix the triangular decomposition \eqref{eq:triangular-2} for $\Uq(\g)$, and consider highest weight $\Uq(\g)$-modules with respect to this
triangular decomposition.

Let $v_0$ be a highest weight vector in a highest weight  $\Uq(\g)$-module, then for all $i$ and $r$,
\begin{eqnarray}\label{eq:def-hw-vector}
\begin{aligned}
\e_{i,r}\cdot v_0=0,\quad \hh^{\pm}_{i,r}\cdot v_0=\up^{\pm}_{i,r}v_0,\quad \ga^{1/2}\cdot v_0=\up^{1/2}v_0,
\end{aligned}
\end{eqnarray}
for some scalars $\up^{\pm}_{i,r}$ and $\up^{1/2}$, where $\up^{1/2}$ is invertible and $\up^{+}_{i,0}\up^{-}_{i,0}=1$. We define the following formal power series in a variable $x$
\[
\up^{+}_i(x):=\sum_{r=0}^{\infty}\up^{+}_{i,r}x^r, \quad   \up^{-}_i(x):=\sum_{r=0}^{\infty}\up^{-}_{i,-r}x^{-r}, \quad \forall i.
\]

A $\Uq(\g)$-module is said to be at level $0$ if $\gamma$ acts by $\pm\rm{id}$.

By considering the commutation relations of $\hh^{\pm}_{i,r}$,
it is easy to show  \cite{CP0, CP1}
that finite dimensional modules must be at level $0$.
The proof of \cite[Proposition 3.2]{CP0} can be adapted verbatim to prove the following result.
\begin{prop}\label{prop:f.m-hwr}
Every finite dimensional simple $\Uq(\g)$-module  is a level $0$ highest weight module
with respect to the triangular decomposition \eqref{eq:triangular-2}.
\end{prop}

The following theorem is the main result of this section.  Its proof will be given in Section \ref{sect:proof-fd}.

\begin{theo}\label{theo-finite module}
Let $\g=\osp(1|2n)^{(1)}$, $\Sl(1|2n)^{(2)}$ and $\osp(2|2n)^{(2)}$. A simple $\Uq(\g)$-module $V$ is finite dimensional if and only if it can be twisted by some automorphism $\iota_\varepsilon$ into a level $0$ simple highest weight
module satisfying the following conditions:
there exist polynomials $P_i\in\C[x]$ $(i=1, 2, \dots, n)$ with constant term $1$ such that
\begin{eqnarray}\label{eq:hw-polys}
\begin{aligned}
&\up^{+}_i(x) =t_i^{c_i\cdot {\rm deg} P_i}\frac{P_i\left((-1)^{n-i}t_i^{-2c_i}x^{d_i}\right)}{P_i\left((-1)^{n-i}x^{d_i}\right)}=\up^{-}_i(x),
\end{aligned}
\end{eqnarray}
where the equalities should be interpreted as follows: the left side is equal to the middle expression expanded at $0$, and the right side to that expanded at $\infty$.
In the above, $t_i = \left(\sqrt{-1}q^{1/2}\right)^{(\alpha_i, \alpha_i)}$,  and $c_i$ and $d_i$ are defined by
\[\begin{array}{l l}
\g= \osp(1|2n)^{(1)}: & d_i = 1, \quad c_i=\begin{cases} 1, & i\neq n, \\
                                                             2, & i=n;
                                       \end{cases}\\
\g=\Sl(1|2n)^{(2)}: &d_i=c_i=1;\\
\g=\osp(2|2n)^{(2)}:  &d_i=c_i=\begin{cases} 1, & i=n, \\
                                                             2, & i\neq n.
                                       \end{cases}
\end{array}
\]
\end{theo}

As an immediate corollary of Theorem \ref{theo-finite module}, we have the following result.
\begin{coro}\label{prop:f.m-type}
Every finite dimensional simple  $\Uq(\g)$-module can be obtained from  a level $0$
type-{\bf {1}} or type-{\bf {s}} module  by twisting $\Uq(\g)$ with some automorphism given in \eqref{eq:auto-1}.
\end{coro}


\begin{rem} \label{rem:gradings}
Any non $\Z_2$-graded simple highest weight $\Uq(\g)$-module can be regarded as graded by simply assign a parity to its highest weight vector.
\end{rem}


\subsection{Proof of Theorem \ref{theo-finite module}}\label{sect:proof-fd}
The theorem can be proven directly by using the method of \cite{CP1, CP2}. However, there is an easier approach based on quantum correspondences between affine Lie superalgebras developed in \cite{Z92b, XZ, Z2}. The
quantum correspondences allow one to translate results on finite dimensional simple modules of ordinary quantum affine algebras in \cite{CP1, CP2} to the quantum affine superalgebras under consideration. We will follow the latter approach here.

\subsubsection{Facts on ordinary quantum affine algebras}
Corresponding to each $\g$ in \eqref{eq:g},  we have an ordinary (i.e., non-super) affine Lie algebra $\g'$  given in Table \ref{tbl:g}, which has the same Cartan matrix as $\g$. We denote by $\{\alpha'_1, \dots, \alpha'_n\}$ the set of simple roots realising the Cartan matrix of $\g'$, and normalize the  bilinear form on the weight space of $\g'$ so that  $(\alpha'_n,\alpha'_n)=1$

Let $\x^\pm_{j, r}$, $\hh'^{\pm}_{i,r}$, and $\gamma'^{\pm1/2}$
($1\le i, j\ge n$, \ $r\in\Z$)  be the generators of the quantum affine algebra $\U_t(\g')$ over $\C(t^{1/2})$ (see \cite{CP1, CP2} for details).
Highest weight $\U_t(\g')$-modules are defined in a similar way as  for $\Uq(\g)$ earlier.
A highest weight $\U_t(\g')$ is generated by a highest weight vector $v'_0$, which satisfies
\begin{eqnarray}\label{eq:V'-hwv}
\x^{+}_{i,r}\cdot v'_0=0,\quad \hh'^{\pm}_{i,r}\cdot v'_0=\up'^{\pm}_{i,r}v'_0,\quad \ga'^{1/2}\cdot v'_0={\up'}^{1/2} v'_0,
\end{eqnarray}
where $\up'^{\pm}_{i,r}\in\C$, with ${\up^\prime}^{1/2}\in\C^*$ and ${\up^\prime}^{+}_{i,0}{\up^\prime}^{-}_{i,0}=1$. The module is at level $0$ if $\up'=\pm1$.

Recall that weight modules over $\U_t(\g')$ can always be twisted to type-{\bf 1} modules by automorphisms analogous to \eqref{eq:auto-1}.
The following result is proved in \cite{CP1,CP2}.
\begin{prop}[\cite{CP1,CP2}]\label{prop-fin.dim}
Let $\g'=A_{2n}^{(2)}, B_n^{(1)}, D_{n+1}^{(2)}$.  Every finite dimensional simple $\U_t(\g')$-module is a highest weight module at level $0$.
A level $0$ simple highest weight  $\U_t(\g')$-module of type-{\bf 1} is finite dimensional if and only if there exist polynomials $Q_i\in\C[x]$ ($1\le i\le n$) with constant term $1$ such that
\begin{align}\label{eq:ordinary-hw-polys}
\sum_{r=0}^{\infty}\up'^{+}_{i,r}x^r=t_i^{c_i\cdot {\rm deg} Q_i}\frac{Q_i\left(t_i^{-2c_i}x^{d_i}\right)}{Q_i(x^{d_i})}=\sum_{r=0}^{\infty}\up^{-}_{i,-r}x^{-r},
\end{align}
which holds in the same sense as \eqref{eq:hw-polys}.
Here $t_i=t^{(\alpha'_i,\alpha'_i)/2}$, and the constants $c_i$ and $d_i$ are those defined in Theorem \ref{theo-finite module} for the affine superalgebra $\g$ corresponding to $\g'$ in Table \ref{tbl:g}.
\end{prop}

\subsubsection{Proof of Theorem \ref{theo-finite module}}
With the preparations above, we can now prove Theorem \ref{theo-finite module}. Note that Theorem \ref{lem:iso} is stated for $\hh_{i,r}^\pm$ instead of $\kappa_{i,s}$  and with the $o(i)$ there worked out explicitly. By using Theorem \ref{lem:iso}, we can identify the categories of $\fU_q(\g)$-modules and $\fU_{-q}(\g')$-modules.  Then Theorem \ref{theo-finite module} is equivalent to Proposition \ref{prop-fin.dim} under this identification.  Let us describe this in more detail.

If a $\fU_{-q}(\g')$-module is generated by a $\U_{-q}(\g')$ highest weight vector that is an eigenvector of the $\sigma'_i$,
it restricts to a simple $\U_{-q}(\g')$-module.  All highest weight $\U_{-q}(\g')$-modules can be obtained this way, but note that different $\fU_{-q}(\g')$-modules of this type may restrict to the same $\U_{-q}(\g')$-module.
Also any highest weight $\U_{-q}(\g')$-module $V'$ with a highest weight vector $v'_0$ can be lifted to a
$\fU_{-q}(\g')$-module by endowing $\C(q^{1/2}) v'_0$ with a  $G'$-module structure (there are many possibilities).
The same discussion applies to $\fU_q(\g)$- and $\U_q(\g)$-modules.

Assume that $V'$ is a simple $\fU_{-q}(\g')$-module generated by a $\U_{-q}(\g')$ highest weight vector $v'_0$ such that $\C(q^{1/2}) v'_0$ is the $1$-dimensional trivial $G'$-module. Then $V'$ is finite dimensional if and only if
it is at level $0$ and the scalars $\up'^{\pm}_{i,r}$ (cf. \eqref{eq:V'-hwv}) satisfy the given condition of Proposition \ref{prop-fin.dim} for some monic polynomials $Q_i$ with $t^{1/2}=\sqrt{-1}q^{1/2}$.
By Theorem \ref{lem:iso},  $V'$ naturally admits the $\fU_q(\g)$-action
\[
\fU_q(\g)\otimes V' \longrightarrow V', \quad x\otimes V'\mapsto \varphi(x)v', \quad \forall x\in \fU_q(\g), \ v'\in V'.
\]
It restricts to a simple highest weight $\U_q(\g)$-module at level $0$ such that
\[
\hh^{\pm}_{i,r}\cdot v'_0=\up^{\pm}_{i,r}v'_0, \
\text{\ \ with\ \ } \up^{+}_{i,r}=(-1)^{(n-i)r\epsilon}\up'^{+}_{i,r}.
\]
Clearly, the $\up^{+}_{i,r}$ satisfy the condition given in Theorem \ref{theo-finite module}.

As a $\fU_q(\g)$-module, $V'$ is naturally $\Z_2$-graded.
Recall from \cite{XZ} that there exists an element $u\in G$ which is the grading operator in the sense that $u x u^{-1} = (-1)^{[x]} x$ for all homogeneous $x\in\Uq(\g)$. The even and odd subspaces of $V'$ are then the $\pm 1$-eigenspaces of $u$.

The above arguments go through in the opposite direction, i.e., from
$\fU_q(\g)$-modules to $\fU_{-q}(\g')$-modules.  This proves Theorem \ref{theo-finite module}.

\section*{Acknowledgements}
This research was supported by National Natural Science Foundation of China Grants No. 11301130,  No. 11431010,
and Australian Research Council Discovery-Project Grant DP140103239.

\end{document}